\documentclass[a4paper,11pt]{article}
\usepackage{amsmath, amsfonts, amscd, amssymb, amsthm, enumerate, xypic}

\newcommand{\Mat}{\operatorname{M}}

\newcommand{\id}{\operatorname{id}}
\newcommand{\GL}{\operatorname{GL}}
\newcommand{\Ker}{\operatorname{Ker}}
\newcommand{\End}{\operatorname{End}}
\newcommand{\Vect}{\operatorname{span}}
\newcommand{\im}{\operatorname{Im}}

\newcommand{\tr}{\operatorname{tr}}

\newcommand{\rk}{\operatorname{rk}}

\renewcommand{\setminus}{\smallsetminus}

\def\F{\mathbb{F}}

\def\N{\mathbb{N}}
\def\Z{\mathbb{Z}}

\def\calA{\mathcal{A}}
\def\calB{\mathcal{B}}

\def\calH{\mathcal{H}}

\def\lcro{\mathopen{[\![}}
\def\rcro{\mathclose{]\!]}}

\theoremstyle{definition}
\newtheorem{Def}{Definition}[section]
\newtheorem{Not}[Def]{Notation}

\theoremstyle{plain}
\newtheorem{theo}{Theorem}[section]
\newtheorem{prop}[theo]{Proposition}
\newtheorem{cor}[theo]{Corollary}
\newtheorem{lemma}[theo]{Lemma}
\newtheorem{claim}{Claim}

\theoremstyle{plain}

\theoremstyle{remark}
\newtheorem{Rems}{Remarks}
\newtheorem{Rem}[Rems]{Remark}
\newtheorem{ex}[Rems]{Example}

\title{Products of involutions of an infinite-dimensional vector space}
\author{Cl\'ement de Seguins Pazzis\footnote{Universit\'e de Versailles Saint-Quentin-en-Yvelines, Laboratoire de Math\'ematiques
de Versailles, 45 avenue des Etats-Unis, 78035 Versailles cedex, France}
\footnote{e-mail address: dsp.prof@gmail.com}}

\begin{document}

\thispagestyle{plain}

\maketitle

\begin{abstract}
We prove that every automorphism of an infinite-dimensional vector space over a field is the product of four involutions,
a result that is optimal in the general case. We also characterize the automorphisms that are the product of three involutions.
More generally, we study decompositions of automorphisms into three or four factors with prescribed split annihilating polynomials of degree $2$.
\end{abstract}



\section{Introduction}

\subsection{The problem}

Let $\F$ be a field. In this article, all the vector spaces we consider have $\F$ as ground field unless stated otherwise.
Given a positive integer $n$, it is folklore that an $n$ by $n$ invertible matrix $A \in \GL_n(\F)$
is a product of involutions if and only if $\det A=\pm 1$. Moreover, if that condition holds then $A$
is actually the product of four involutions \cite{Gustafsonetal}, and it is known that three involutions do not suffice in general
(classically, if $\lambda$ is a nonzero scalar such that $\lambda^4 \neq  1$ and $\lambda^n=\pm 1$, then the matrix $\lambda I_n$ is not the product of three involutions).
The products of two involutions in $\GL_n(\F)$ are the matrices that are similar to their inverse (see \cite{Djokovic}, \cite{HoffmanPaige} and
\cite{Wonenburger}). However, there is no neat classification of the invertible matrices that are products of three involutions, and there is
little hope of ever finding one (see \cite{Ballantine,Liu} and the very recent \cite{dSP3involutions}).

To our knowledge, the corresponding issue in the infinite-dimensional setting has been entirely neglected
although it was cited as an interesting open issue in the end section of \cite{Gustafsonetal}.
So, let $V$ be an infinite-dimensional vector space, and consider the ring $\End(V)$
of all endomorphisms of $V$, and its group $\GL(V)$ of units, i.e.\ the automorphisms of $V$.
Of course, every product of (linear) involutions of $V$ belongs to $\GL(V)$.
We aim at proving that every element of $\GL(V)$ is a product of involutions. Better still, we will find
the minimal integer $p$ such that every element of $\GL(V)$ is the product of $p$ involutions.

There is a corresponding problem in the group of units of the algebra of all bounded operators of an infinite-dimensional complex Hilbert space $H$:
it was shown by Radjavi \cite{Radjavi} that every invertible operator in that algebra is the product of
seven involutions; moreover, four involutions do not suffice in general
(it is not difficult to prove that, for any nonzero complex number $z$ with modulus different from $1$, the element $z\cdot \id_H$
is not the product of four involutions in the algebra $\calB(H)$). It is not known whether five involutions suffice, and
no advance has been made ever since in this problem. A few years earlier, Halmos and Kakutani \cite{HalmosKakutani} famously
proved that every unitary operator on $H$ is the product of four symmetries (i.e.\ of unitary involutions), and that three symmetries do not suffice in general.
However, those problems are profoundly different from ours because they involve structures from analysis, whereas
the problem we deal with here is purely algebraic.

\subsection{Main results}

Here is one of our main results:

\begin{theo}\label{theo4invol}
Let $V$ be an infinite-dimensional vector space.
Every automorphism of $V$ is the product of four involutions.
\end{theo}

We will also prove that three involutions suffice if $|\F| \leq 3$, but not otherwise,
and in general we will completely characterize the automorphisms that split into the product of three involutions
(see Theorem \ref{theo3invol} below).

Our methods actually allow a generalization of the above theorem, and a bit of additional definition will help.
We denote by $\F[t]$ the algebra of all polynomials with one indeterminate $t$
over $\F$. An element $x$ of an $\F$-algebra $\calA$ is called \textbf{quadratic}
whenever it is annihilated by some polynomial $p \in \F[t]$ with degree $2$.
A polynomial $p \in \F[t]$ will be called \textbf{non-derogatory} if $p(0) \neq 0$.
Note that any element of $\calA$ that is annihilated by such a polynomial
is invertible.

The involutions in $\calA$ are the quadratic elements that are annihilated by $t^2-1$.
An element of $\calA$ is called \textbf{unipotent with index $2$} whenever it is annihilated by $(t-1)^2$, i.e.\
when it reads $1_\calA+u$ for some $u \in \calA$ such that $u^2=0_\calA$.

Let $p_1,\dots,p_n$ be non-derogatory polynomials of $\F[t]$. An element $x$ of $\calA$ is called a
\textbf{$(p_1,\dots,p_n)$-product} when it splits into $x=a_1\cdots a_n$ for some list $(a_1,\dots,a_n) \in \calA^n$
such that $p_i(a_i)=0$ for all $i \in \lcro 1,n\rcro$. Note that such an element must then be invertible.
For example, an automorphism of $V$ is the product of $n$ involutions in $\End(V)$ if and only if
it is a $(t^2-1,\dots,t^2-1)$-product (with $n$ copies of $t^2-1$).
Note also that the set of all $(p_1,\dots,p_n)$-products is stable under conjugation. Noting that
$a_1^{-1}(a_1 \cdots a_n)a_1=a_2\cdots a_n a_1$ for any invertible elements $a_1,\dots,a_n$,
we deduce that the $(p_1,\dots,p_n)$-products are the $(p_2,\dots,p_n,p_1)$-products.

Now, we can state a generalization of Theorem \ref{theo4invol}:

\begin{theo}\label{theo4general}
Let $p_1,\dots,p_4$ be four split non-derogatory polynomials with degree $2$ in $\F[t]$,
and $V$ be an infinite-dimensional vector space. Then every automorphism of $V$ is a
$(p_1,p_2,p_3,p_4)$-product.
\end{theo}

In particular, every automorphism of $V$ is the product of four involutions, but also of four
unipotent automorphisms with index $2$, and also of three involutions and one unipotent automorphism with index $2$.
And so on. The corresponding result for decompositions into sums was proved recently \cite{Breaz,dSPidempotentinfinite1,Shitov}.

Next, we will see that three factors do not suffice in general, and we will almost completely classify
the $(p_1,p_2,p_3)$-products. The main obstruction for being a $(p_1,p_2,p_3)$-product comes
from the existence of a ``large" eigenvalue:

\begin{Def}
Let $u \in \End(V)$.
A scalar $\lambda$ is called a \textbf{dominant eigenvalue} of $u$ if
$\rk(u-\lambda\,\id_V)<\dim V$.
\end{Def}

In that case, $\lambda$ is truly an eigenvalue of $u$, i.e.\ $\Ker(u-\lambda \,\id_V) \neq \{0\}$,
and it is the sole dominant eigenvalue of $u$ (indeed, given distinct scalars $\alpha$ and $\beta$ in $\F$, the equality
$(\beta-\alpha)\,\id_V=(u-\alpha\,\id_V)-(u-\beta\,\id_V)$ yields $V=\im(u-\alpha\id_V)+\im(u-\beta\id_V)$, and hence
$\dim V \leq \rk(u-\alpha\id_V)+\rk(u-\beta \id_V)$ ; here $V$ has infinite dimension and hence at least one of
$u-\alpha \id_V$ and $u-\beta \id_V$ has rank $\dim V$).

A mundane example is the one where $u=\lambda \,\id_V$.
Note that if $u$ has a dominant eigenvalue $\lambda$ and $u$ is an automorphism of $V$ then $\lambda$ is non-zero.

As we are about to see, there are limitations on the possible dominant eigenvalues of a $(p_1,p_2,p_3)$-product.
To formulate such limitations, recall that, for a  monic polynomial $p \in \F[t]$ with degree $n$,
its norm is defined as $(-1)^n p(0)$, and its trace is defined as the opposite of the coefficient of $p$ on $t^{n-1}$.
The \textbf{norm} and \textbf{trace} of a nonzero polynomial $p$ with leading coefficient $\alpha$ are defined as those of
$\alpha^{-1}p$, and denoted by $N(p)$ and $\tr(p)$, respectively.

\begin{theo}\label{dominanteigenvalueCN}
Let $V$ be an infinite-dimensional vector space.
Let $p_1,p_2,p_3$ be split non-derogatory polynomials of $\F[t]$ with degree $2$.
Let $u \in \GL(V)$ have a dominant eigenvalue $\lambda$, and
assume that $u$ is a $(p_1,p_2,p_3)$-product.
Then $\lambda$ is a $(p_1,p_2,p_3)$-product in the algebra $\F$ or $\lambda^2=N(p_1)N(p_2)N(p_3)$.
\end{theo}

\begin{Def}
Let $p_1,p_2,p_3$ be split non-derogatory polynomials with degree $2$ over $\F$.
A scalar $\lambda \in \F$ is called \textbf{acceptable} for the triple $(p_1,p_2,p_3)$ when
it is a $(p_1,p_2,p_3)$-product in the algebra $\F$ or $\lambda^2=N(p_1)N(p_2)N(p_3)$.
\end{Def}

For example, in the case of products of three involutions, the acceptable scalars are the fourth roots of the unity.

In particular, whenever $|\F|>5$ or $|\F|=4$, there exists $\lambda \in \F \setminus \{0\}$ such that $\lambda^4 \neq 1$,
and hence $\lambda \id_V$ is not the product of three involutions.

\vskip 3mm
In contrast, if $u$ has no dominant eigenvalue, then everything works fine for decompositions into three factors. This is the major result
of the present article:

\begin{theo}\label{theo3}
Let $V$ be an infinite-dimensional vector space, and let $p_1,p_2,p_3$ be split non-derogatory polynomials
of $\F[t]$ with degree $2$.
Let $u$ be an automorphism of $V$ with no dominant eigenvalue. Then $u$ is a $(p_1,p_2,p_3)$-product.
\end{theo}

There is also a quasi-converse statement for Theorem \ref{dominanteigenvalueCN}:

\begin{theo}\label{dominanteigenvalueCS}
Let $V$ be an infinite-dimensional vector space, and let $p_1,p_2,p_3$ be split non-derogatory polynomials
of $\F[t]$ with degree $2$.
Let $u$ be an automorphism of $V$. Assume that $u$ has a dominant eigenvalue $\lambda$ that is acceptable for $(p_1,p_2,p_3)$
and such that $u-\lambda\,\id_V$ does not have finite rank.
Then $u$ is a $(p_1,p_2,p_3)$-product.
\end{theo}

Hence, in order to have a full characterization of the $(p_1,p_2,p_3)$-products in $\GL(V)$, it remains
to understand when an automorphism which is the sum of $\lambda \id_V$
with a finite-rank endomorphism is a $(p_1,p_2,p_3)$-product. We will show in Section \ref{finiteranksection} that
this amounts to determine, given a scalar $\lambda$ that is acceptable for $(p_1,p_2,p_3)$,
for which square matrices $A \in \Mat_n(\F)$ there exists an integer $q \geq 0$ such that $A \oplus \lambda I_q$ is a
$(p_1,p_2,p_3)$-product: it turns out that this problem is open for general values of $(p_1,p_2,p_3)$.
Nevertheless, for a few specific values of the triple $(p_1,p_2,p_3)$, it has recently been solved (see \cite{dSP3involutions}).
This leads to the following three results, which encompass all decompositions that involve only
involutions and unipotent automorphisms with index $2$:

 \begin{theo}\label{theo3invol}
Let $V$ be an infinite-dimensional vector space.
Let $u \in \GL(V)$. Then $u$ is the product of three involutions if and only if none of the following conditions holds:
\begin{enumerate}[(i)]
\item There is a scalar $\lambda$ such that $\rk(u-\lambda \id_V)<\dim V$ and $\lambda^4 \neq 1$.
\item There is a scalar $\lambda$ such that $u-\lambda \id_V$ has finite rank, $\lambda^4=1$ and the determinant of the
automorphism of $\im(u-\lambda \id_V)$ induced by $u$ does not belong to
$\{\varepsilon \lambda^k \mid k \in \N, \; \varepsilon \in \{-1,1\}\}$.
\end{enumerate}
Moreover, $u$ is the product of three involutions if and only if it is the product of one involution
and two unipotent automorphisms with index $2$.
\end{theo}

\begin{cor}
If $|\F| \leq 3$ then every automorphism of an infinite-dimensional vector space over $\F$
is the product of three involutions.
\end{cor}

In the finite-dimensional case, it is known that the product of two unipotent automorphisms with index $2$
is always the product of two involutions (but the converse fails), see \cite{dSPproductquad}. To our knowledge,
the validity of this statement in the infinite-dimensional setting remains an open problem.

\begin{theo}\label{theo3unipotent}
Let $V$ be an infinite-dimensional vector space.
Let $u \in \GL(V)$. Then $u$ is the product of three unipotent automorphisms with index $2$ if and only if none of the following conditions holds:
\begin{enumerate}[(i)]
\item There is a scalar $\lambda$ such that $\rk(u-\lambda \id_V)<\dim V$ and $\lambda^2 \neq 1$.
\item The endomorphism $u-\id_V$ has finite rank and the determinant of the automorphism of $\im(u-\id_V)$ induced by $u$ is not $1$.
\item The endomorphism $u+\id_V$ has finite rank and the determinant of the automorphism of $\im(u+\id_V)$ induced by $u$ is neither $1$ nor $-1$.
\end{enumerate}
\end{theo}

\begin{theo}\label{theo3mixed}
Let $V$ be an infinite-dimensional vector space.
Let $u \in \GL(V)$. Then $u$ is the product of one unipotent automorphism with index $2$ and two involutions
if and only if none of the following conditions holds:
\begin{enumerate}[(i)]
\item There is a scalar $\lambda$ such that $\rk(u-\lambda \id_V)<\dim V$ and $\lambda^2 \neq 1$.
\item There is a scalar $\lambda \in \{1,-1\}$ such that $u-\lambda \id_V$ has finite rank and the determinant of
the automorphism of $\im(u-\lambda \id_V)$ induced by $u$ is neither $1$ nor $-1$.
\end{enumerate}
\end{theo}

With the above, little mistery remains on the decompositions into products of three quadratic automorphisms.
We would still like to point to some open problems. First, we wonder whether the above could be generalized to
irreducible annihilating polynomials. For example, for infinite-dimensional real vector spaces it would be nice to know
whether every automorphism splits into the product of several automorphisms $v$ such that $v^2=-\id_V$ (i.e.\ quarter-of-turns)
and what is the minimal number of factors required in general. In fact, one can show that in Theorem \ref{theo3} (respectively, Theorem \ref{theo4general}), we can allow
one (respectively, two) of the polynomials to be irreducible.

Also, there remains the issue of classifying the $(p_1,p_2)$-products.
For example, the product of two involutions must be similar to its inverse, and the converse is true in finite-dimensional vector spaces:
but is it true in the infinite-dimensional case? The lack of any sort of canonical form for an endomorphism of an infinite-dimensional vector space
makes us worry that this could be a very difficult problem.

\subsection{Strategy}

The above results are very similar to the ones of \cite{dSPidempotentinfinite1,dSPidempotentinfinite2}
for decompositions into sums. Unsurprisingly, we will follow a similar strategy, and at times we will even
be able to borrow some lemmas! However, there are a few major technical differences, and we will discuss some of them here.

Key to \cite{dSPidempotentinfinite1,dSPidempotentinfinite2} was the notion of an elementary endomorphism of $V$.
An elementary endomorphism of $V$ was defined there as a direct sum of right-shifts, a right-shift being a linear mapping $v$
on a linear subspace $W$ of $V$ equipped with a basis $(e_k)_{k \geq 0}$ indexed over the non-negative integers, such that
$v(e_k)=e_{k+1}$ for all $k$. It is quite easy to see that such an endomorphism splits into the sum of two square-zero endomorphisms, and
can also accommodate various similar decompositions. In \cite{Shitov}, Shitov
proved that any endomorphism of an infinite-dimensional vector space is the sum of two elementary endomorphisms,
yielding the equivalent of Theorem \ref{theo4general} for decompositions into sums.

In this work, the relevant notion of elementary endomorphism is different: it is related to \emph{bilateral shifts} instead of right-shifts.
Let us be more precise. Denote, respectively, by $\N$ and $\Z$ the set of all non-negative integers and the set of all integers.
Denote by $\F[t,t^{-1}]$ the algebra of Laurent polynomials with one indeterminate $t$.
An automorphism $u$ of $V$ endows it with a structure of $\F[t,t^{-1}]$-module,
so that $t\cdot x=u(x)$ for all $x \in V$: We use the notation $V^u$ when we speak of $V$ as an $\F[t,t^{-1}]$-module.
Conversely, given an $\F[t,t^{-1}]$-module $W$, the mapping $x \mapsto t\cdot x$ is an automorphism of the $\F$-vector space $W$.
The submodules of $V^u$ are the linear subspaces of $V$ that are stable under $u$ and $u^{-1}$ (put differently, they are the linear subspaces
$W$ of $V$ such that $u(W)=W$).

Let us view $\F[t,t^{-1}]$ as a module over itself.
Then the automorphism $p(t) \mapsto tp(t)$ maps $t^k$ to $t^{k+1}$ for all $k \in \Z$.
An automorphism $u$ of a vector space $V$ will be called \textbf{super-elementary} when $V^u$ is isomorphic to
$\F[t,t^{-1}]$, which is equivalent to the existence of a basis $(e_k)_{k \in \Z}$ of the vector space $V$
such that $u(e_k)=e_{k+1}$ for all $k \in \Z$.
We will say that $u$ is \textbf{elementary} whenever $u$ is a direct sum of super-elementary automorphisms,
i.e.\ $V^u$ is a free $\F[t,t^{-1}]$-module.

Here, elementary automorphisms are very handy because they admit decompositions into a very small number of factors:

\begin{theo}\label{theo2}
Let $p,q$ be split non-derogatory polynomials in $\F[t]$ with degree $2$.
Then every elementary automorphism of a vector space is a $(p,q)$-product.
\end{theo}

The corresponding result for sums (in which bilateral shifts are replaced with right-shifts)
was easy to see. Here, things are far less obvious!

In \cite{Shitov}, Shitov proved that every endomorphism of an infinite-dimensional vector space is the sum of
two endomorphisms that are direct sums of right-shifts, and he deduced the equivalent of Theorem \ref{theo4general} for sums.
An interesting open question is whether every automorphism of an infinite-dimensional vector space is the product of two
elementary ones. Note that if this were true then as an application every automorphism of a vector space $V$ with uncountable dimension
would be a commutator in the group $\GL(V)$! Indeed, the inverse of an elementary automorphism is an elementary one, and if
$V$ has uncountable dimension then it is easily seen that all elementary automorphisms are conjugate in $\GL(V)$.

\begin{Def}
Let $V$ be an $\F$-vector space, and $p \in \F[t]$ be a non-derogatory split polynomial with degree $2$.
Let $u \in \GL(V)$ and $v \in \GL(V)$. We say that $u$ is \textbf{$p$-adjacent to $v$} whenever there exists
$a \in \GL(V)$ such that $p(a)=0$ and $au=v$.
\end{Def}

\begin{Not}
Given  a non-derogatory polynomial $p(t)=a_n\, t^n+\cdots +a_0$ of degree $n$, we denote by
$$p^\sharp(t) :=t^n p(1/t)=a_0\, t^n+\cdots + a_n$$
its reciprocal polynomial (which is also non-derogatory).
Note that $(p^\sharp)^\sharp=p$.
\end{Not}

Noting that an automorphism $a \in \GL(V)$ is annihilated by $p$ if and only if its inverse is annihilated by $p^\sharp$,
we deduce that, given non-derogatory polynomials $p_1,\dots,p_n$, an element of $\GL(V)$
is a $(p_1,\dots,p_n)$-product if and only if it is $p_1^\sharp$-adjacent to a $(p_2,\dots,p_n)$-product.
Hence, in order to obtain Theorem \ref{theo3}, it suffices to combine Theorem \ref{theo2} with the following result,
whose proof will take up the larger part of the present article:

\begin{theo}\label{adjacencytheo}
Let $u$ be an automorphism of an infinite-dimensional vector space, with no dominant eigenvalue.
Then for any split non-derogatory polynomial $p \in \F[t]$ with degree $2$,
the automorphism $u$ is $p$-adjacent to an elementary automorphism.
\end{theo}

Note that this result is optimal. Indeed, let $a \in \GL(V)$ be annihilated by some non-derogatory split monic polynomial
$p(t)=(t-x)(t-y)$ with degree $2$, and let $u \in \GL(V)$ have a dominant eigenvalue $\lambda$.
Using $\im(a-x\id_V) \subset \Ker(a-y \id_V)$,
we find that if $\dim \Ker(a-y \id_V)<\dim V$ then $x$ is a dominant eigenvalue of $a$,
and hence the rank theorem shows that $\dim V=\dim \Ker(a-x \id_V)$.
Hence, one of the kernels $\Ker(a-x \id_V)$ and $\Ker(a-y \id_V)$ has dimension $\dim V$.
Then such a subspace does not have trivial intersection with $\Ker(u-\lambda \id_V)$ otherwise the rank theorem would yield
that it is isomorphic to a subspace of $\im(u-\lambda \id_V)$.
Hence, $\Ker(u-\lambda \id_V)$ has a nonzero common vector with one of the spaces $\Ker(a-x \id_V)$ or $\Ker(a-y \id_V)$,
and it ensues that one of $\lambda x$ and $\lambda y$
is an eigenvalue of $au$. However, it is obvious that
no elementary automorphism has an eigenvalue.

We finish with a very useful remark:

\begin{Rem}\label{directsumremark}
Let $p_1,\dots,p_r$ be non-derogatory polynomials with degree $2$ in $\F[t]$.
Then being a $(p_1,\dots,p_r)$-product is stable under taking direct sums.
Let us explain this in more details: let $u$ be an automorphism of a vector space $V$.
Assume that we have a splitting $V=\underset{i \in I}{\bigoplus} V_i$ in which $u(V_i)=V_i$ for all $i \in I$,
and denote by $u_i$ the induced automorphism of $V_i$. Assume also that, for all $i \in I$,
the endomorphism $u_i$ splits into $u_i=\underset{k=1}{\overset{r}{\prod}} u_{i,k} $ where $u_{i,k}\in \GL(V_i)$ and $p_k(u_{i,k})=0$
for all $k \in \lcro 1,r\rcro$.
Then by setting $u^{(k)}:=\underset{ i \in I}{\bigoplus} u_{i,k} \in \GL(V)$, we see that
$u=\underset{k=1}{\overset{r}{\prod}} u^{(k)}$ and $p_k(u^{(k)})=0$ for all $k \in \lcro 1,r\rcro$.
\end{Rem}

\subsection{Structure of the article}

The remainder of the article is laid out as follows.
In Section \ref{CNsection}, we prove Theorem \ref{dominanteigenvalueCN}.
This is done in two steps: first, we determine which scalar multiples of the identity are
$(p_1,p_2,p_3)$-products; then, we prove that decomposing an automorphism $u$ with a dominant eigenvalue
$\lambda$ into a product of three quadratic automorphisms yields a non-trivial subspace that includes $\im(u-\lambda \id_V)$
and is invariant under all the factors at hand (the Invariant Subspace Lemma). From those results, we will easily derive Theorem \ref{dominanteigenvalueCN}.
With the same line of reasoning, we will find in Section \ref{finiteranksection} that characterizing the $(p_1,p_2,p_3)$-products
among the operators of the form $\lambda \id_V+w$, with $\lambda \in \F$ and $w$ of finite rank, amounts to solving a problem in the finite-dimensional case.
In particular, the case when each $p_i$ equals $t^2-1$ or $(t-1)^2$ will be completely solved thanks to recent decomposition results in the stable
general linear group \cite{dSP3involutions}.

The remaining sections are devoted to sufficient conditions for splitting an automorphism into quadratic ones.
First, elementary automorphisms are studied in Section \ref{elementarysection}: there, we devise ways to recognize super-elementary
automorphisms, and we obtain Theorem \ref{theo2} as a consequence.
The next three sections are devoted to the proof of Theorem \ref{theo3}:
in Section \ref{stratificationsection}, we develop the notion of a stratification of an $\F[t,t^{-1}]$-module
(which closely mimics the one
for $\F[t]$-modules featured in \cite{dSPidempotentinfinite1,dSPidempotentinfinite2}) and we prove that, assuming the existence
of a \emph{semi-good} stratification of $V^u$, the automorphism $u$ is $p$-adjacent to an elementary automorphism
(for any non-derogatory split polynomial $p$ with degree $2$).
In Section \ref{uncountablesection}, we prove that such a stratification exists whenever $u$ has no dominant eigenvalue and $V$ has
uncountable dimension, thereby completing the proof of Theorem \ref{theo3} for such spaces.
In Section \ref{countablesection}, we complete the proof of Theorem \ref{theo3} for spaces with infinite countable dimension:
thanks to a result that was proved in \cite{dSPidempotentinfinite2}, we limit the discussion to the case when $V^u$
includes a free submodule of rank $1$, and in that case we complete the proof by using a new method.

In the ultimate section, we will easily derive Theorems \ref{theo4general} and \ref{dominanteigenvalueCS} from Theorem~\ref{theo3}.

\section{Necessary conditions for being a $(p_1,p_2,p_3)$-product}\label{CNsection}

\subsection{The case of scalar multiples of the identity}

\begin{prop}\label{homothetieprop1}
Let $V$ be an infinite-dimensional vector space, and let $\lambda \in \F \setminus \{0\}$.
Let $p_1,p_2,p_3$ be non-derogatory split polynomials with degree $2$ over $\F$,
and let $\lambda \in \F$ be acceptable for $(p_1,p_2,p_3)$.
Then $\lambda\,\id_V$ is a $(p_1,p_2,p_3)$-product.
\end{prop}

\begin{proof}
Without loss of generality, we assume that $p_1,p_2,p_3$ are all monic, and we
set $\alpha:=N(p_1)$, $\beta:=N(p_2)$ and $\gamma:=N(p_3)$.

If $\lambda$ is a $(p_1,p_2,p_3)$-product, we split it up into
$\lambda=x_1x_2x_3$ where $p_i(x_i)=0$ for all $i \in \lcro 1,3\rcro$, and it is then obvious
from writing $\lambda\,\id_V=(x_1\,\id_V) (x_2\id_V) (x_3\id_V)$ that $\lambda\,\id_V$ is a $(p_1,p_2,p_3)$-product
in $\End(V)$.

Assume now that $\lambda^2=\alpha\beta\gamma$. We shall prove that the scalar matrix
$\lambda I_2$ is a $(p_1,p_2,p_3)$-product in the algebra $\Mat_2(\F)$ of all $2$ by $2$ matrices with entries in $\F$.
We note that this can be deduced directly from the general classification of $(p_2,p_3)$-products featured in \cite{dSPproductquad} by noting that
the companion matrix $D$ of $p_1$ is such that $\lambda D^{-1}$ is a $(p_2,p_3)$-product, but we will give a simple direct proof.
Let us split $p_1(t)=(t-x)(t-y)$ with $(x,y)\in (\F\setminus \{0\})^2$.
Write also $\beta^{-1}t^2p_2(t^{-1})=t^2-\mu t+\beta^{-1}$
and $\gamma^{-1}t^2p_3(t^{-1})=t^2-\nu t+\gamma^{-1}$.
The matrices
$$B':=\begin{bmatrix}
0 & -\beta^{-1} \\
1 & \mu
\end{bmatrix} \quad \text{and} \quad C':=\begin{bmatrix}
\nu & \lambda^{-1}x \\
-x^{-1}\gamma^{-1}\lambda & 0
\end{bmatrix}$$
are invertible, and the Cayley-Hamilton theorem reveals that they are respectively annihilated by $p_2^\sharp$ and $p_3^\sharp$.
Therefore, $B:=(B')^{-1}$ and $C:=(C')^{-1}$ are respectively annihilated by $p_2$ and $p_3$.

Since $\lambda^2=\alpha\beta\gamma$ and $x^{-1}\alpha=y$, we find
$$\lambda\,C^{-1}B^{-1}=\lambda C'B'=\begin{bmatrix}
x & ? \\
0 & y
\end{bmatrix}.$$
Hence, $A:=\lambda\,C^{-1}B^{-1}$ is annihilated by $p_1$, and we conclude that
$\lambda\, I_2=ABC$ is a $(p_1,p_2,p_3)$-product in $\Mat_2(\F)$. It follows that
for every $2$-dimensional vector space $P$ over $\F$,
the automorphism $\lambda \id_P$ is a $(p_1,p_2,p_3)$-product in the algebra $\End(P)$.

Finally, since $V$ is infinite-dimensional we can split it up into $V=\underset{i \in I}{\bigoplus} Q_i$
where each $Q_i$ is a $2$-dimensional vector space. Then $\lambda\,\id_{Q_i}$ is a $(p_1,p_2,p_3)$-product
for all $i \in I$, and
by Remark \ref{directsumremark} we conclude that $\lambda\,\id_V$ is a $(p_1,p_2,p_3)$-product.
\end{proof}

The converse of Proposition \ref{homothetieprop1} holds thanks to the following general result:

\begin{prop}\label{homothetieprop2}
Let $\calA$ be a non-trivial $\F$-algebra, and let $\lambda \in \F$.
If $\lambda\cdot 1_\calA$ is a $(p_1,p_2,p_3)$-product in $\calA$, then $\lambda$ is acceptable for $(p_1,p_2,p_3)$.
\end{prop}

To prove this result, a few basic facts on quadratic elements will help.
Let $p$ be a monic polynomial with degree $2$ in $\F[t]$, and let $a$ be an element of an $\F$-algebra $\calA$ such that
$p(a)=0$. We define $a^\star_p:=(\tr p)\,1_\calA-a$ and we note that $aa_p^\star=a_p^\star a=N(p)\,1_\calA$.
In particular, if $p$ is non-derogatory then $a$ is invertible and $a_p^\star=N(p)\, a^{-1}$.
We recall the following basic results (see lemmas 1.2 and 1.3 from \cite{dSPproductquad}), which we quickly reprove:

\begin{lemma}\label{commutelemma1}
Let $a,b$ be elements of an $\F$-algebra $\calA$, and let $p,q$ be monic polynomials with degree $2$
such that $p(a)=q(b)=0$. Then both $a$ and $b$ commute with $ab_q^\star+b a_p^\star$
\end{lemma}

\begin{proof}
By expanding, we obtain
$$ab_q^\star+b a_p^\star=\tr(q)\,a+\tr(p)\, b-ab-ba=b_q^\star a+a_p^\star b,$$
and hence
$$b\bigl(ab_q^\star+b a_p^\star\bigr)=b\bigl(b_q^\star a+a_p^\star b\bigr)=N(q)\,a+ba_p^\star b$$
and
$$\bigl(ab_q^\star+b a_p^\star\bigr)\,b=N(q)\,a+b a_p^\star b.$$
We deduce that $b$ commutes with $u:=ab_q^\star+b a_p^\star$. Symmetrically, $a$ commutes with $u$.
\end{proof}

\begin{lemma}\label{commutelemma2}
Let $a,b$ be elements of an $\F$-algebra $\calA$, and let $p,q$ be non-derogatory monic polynomials with degree $2$
such that $p(a)=q(b)=0$. Then both $a$ and $b$ commute with $ab+N(p)N(q)(ab)^{-1}$.
\end{lemma}

\begin{proof}
Let us use the same notation regarding $a_p^\star$ and $b_q^\star$ as in Lemma \ref{commutelemma1}.
Noting that $q(b_q^\star)=0$, we deduce from Lemma \ref{commutelemma1} that
both $a$ and $b_q^\star$ commute with $ab+b_q^\star a_p^\star =ab+N(p)N(q)\, b^{-1}a^{-1}=ab+N(p)N(q)\,(ab)^{-1}$.
We conclude by noting that $b=(\tr q)\,1_\calA-b_q^\star$.
\end{proof}

We are now ready to prove Proposition \ref{homothetieprop2}:

\begin{proof}[Proof of Proposition \ref{homothetieprop2}]
Let $a,b,c$ belong to $\calA$, be annihilated respectively by $p_1,p_2,p_3$ and satisfy
$abc=\lambda\cdot 1_\calA$.

Without loss of generality, we assume that $p_1,p_2,p_3$ are all monic, and we
set $\alpha:=N(p_1)$, $\beta:=N(p_2)$ and $\gamma:=N(p_3)$.
Assuming that $\lambda^2  \neq \alpha\beta\gamma$, we shall prove that $\lambda$ is a $(p_1,p_2,p_3)$-product in $\F$.
By Lemma \ref{commutelemma2}, we know that both $b$ and $c$ commute with
$u:=(bc)+\beta \gamma (bc)^{-1}$.
Yet, as $ bc=\lambda\,a^{-1}$ we get
$$u=\lambda a^{-1}+\beta \gamma \lambda^{-1} a=
\lambda \alpha^{-1} a_{p_1}^\star +\beta \gamma \lambda^{-1} a=\lambda \alpha^{-1} \tr(p_1)\,1_\calA
+(-\lambda \alpha^{-1}+\beta \gamma \lambda^{-1})\, a.$$
Since $\lambda^2  \neq \alpha\beta\gamma$, we have
$-\lambda \alpha^{-1}+\beta \gamma \lambda^{-1} \neq 0$ and we deduce that $b$ and $c$ commute with $a$.
Then $\lambda\,1_\calA=bac$, and hence applying what we have just found in this new setting yields that $c$ commutes with $b$.

We are about to conclude. Denote by $\calB$ the subalgebra of $\calA$ generated by $a,b,c$. Since $a,b,c$ have non-zero annihilating polynomials
and are commuting elements, we see that $\calB$ is finite-dimensional. Set $\widetilde{a} : x \in \calB \mapsto ax$,
$\widetilde{b} : x \in \calB \mapsto bx$ and $\widetilde{c} : x \in \calB \mapsto cx$: these are
commuting endomorphisms of the (non-zero) finite-dimensional $\F$-vector space $\calB$ that are respectively annihilated by the split polynomials $p_1,p_2,p_3$,
and in particular they are commuting triangularizable endomorphisms.
Hence, $\widetilde{a}$, $\widetilde{b}$ and $\widetilde{c}$ have a common eigenvector $x \in \calB \setminus \{0\}$, and more precisely
$ax=\omega_1\,x$, $bx=\omega_2\,x$ and $cx=\omega_3\,x$ for respective roots $\omega_1,\omega_2,\omega_3$ of $p_1,p_2,p_3$.
It follows that $\lambda x=(\omega_1 \omega_2\omega_3) x$, and hence $\lambda=\omega_1 \omega_2\omega_3$.
Therefore, $\lambda$ is a $(p_1,p_2,p_3)$-product in $\F$.

%
\end{proof}

\subsection{The invariant subspace lemma}

\begin{lemma}[Invariant Subspace Lemma]\label{invariantsubspacelemma}
Let $V$ be an infinite-dimensional vector space.
Let $a,b,c$ be quadratic automorphisms of $V$. Let $\lambda \in \F \setminus \{0\}$ and $w \in \End(V)$ be such that
$\lambda\,\id_V+w=abc$. Let $W$ be a linear subspace of $V$ that includes $\im w$.
Then there exists a linear subspace $\overline{W}$ of $V$
that includes $W$, is stable under $a$, $b$ and $c$, and satisfies
$\dim \overline{W}\leq 8\,\dim W$.
In particular, $\dim \overline{W}<\dim V$ whenever $V$ is infinite-dimensional and $\dim W < \dim V$.
\end{lemma}

\begin{proof}
Replacing $a$ with $\lambda^{-1}a$ and $w$ with $\lambda^{-1}w$, we can assume that $\lambda=1$.
Let us set
$$\overline{W}:=W+a(W)+b(W)+c(W)+\underset{(e,f) \in \{a,b,c\}^2}{\sum} ef(W).$$
We start by proving that
\begin{equation}\label{reduction}
\overline{W}=W+a(W)+b(W)+c(W)+ba(W)+cb(W)+ac(W)+ca(W),
\end{equation}
which will yield $\dim \overline{W} \leq 8 \dim W$.

Throughout, we denote by $\calH$ the right-sided ideal of $\End(V)$ consisting of its
elements whose range is included in $W$. Note that $w \in \calH$.

First of all, since $a$ is quadratic, we have $a^2(W) \subset W+a(W)$,
and likewise $b^2(W) \subset W+b(W)$ and $c^2(W) \subset W+c(W)$.

Next, since $c$ is quadratic we find $c^{-1} \in \Vect(\id_V,c)$, and hence
\begin{equation}\label{ab}
ab=(abc)c^{-1}=c^{-1}+w c^{-1} \in \Vect(\id_V,c)+\calH.
\end{equation}
In particular $ab(W) \subset W+c(W)$.
Likewise, since $a$ is quadratic,
\begin{equation}\label{bc}
bc=a^{-1}(abc)=a^{-1}+a^{-1}w \in \Vect(\id_V,a)+a\calH+\calH,
\end{equation}
and hence $bc(W) \subset W+a(W)$.
This yields identity \eqref{reduction}.

In the next part of the proof, we show that $a$ maps all spaces $ba(W)$, $cb(W)$, $ac(W)$ and $ca(W)$ into
$\overline{W}$, which will show that $a$ stabilizes $\overline{W}$.
First of all $a(ac)=a^2 c \in \Vect(c,ac)$ because $a$ is quadratic. Hence, $a(ac(W)) \subset \overline{W}$.
Next, by right-multiplying by $a$ in \eqref{ab}, we find
$$aba \in \Vect(a,ca)+\calH,$$
and hence $a(ba(W)) \subset W+a(W)+ca(W) \subset \overline{W}$.
Finally, starting from $\id_V=abc-w$, we find
$$ac=a(abc-w)c\in a^2bc^2+a\calH.$$
Since $a$ and $c$ are quadratic and $abc=\id_V+w$, this yields
$$ac \in \Vect(abc,ab,bc,b)+a\calH \subset \Vect(\id_V,ab,bc,b)+a \calH+\calH.$$
It follows that
$$aca \in \Vect(a,aba,bca,ba)+a \calH+\calH \subset \Vect(\id_V,a,bca,ca,ba)+a \calH+\calH$$
and
$$acb \in \Vect(b,ab^2,bcb,b^2)+a \calH+\calH \subset \Vect(\id_V,a,b,ab,bcb)+a \calH+\calH.$$
By multiplying by $a$ or $b$ on the right in \eqref{bc}, we find
$$bca \in \Vect(a,a^2)+a\calH+\calH \subset \Vect(\id_V,a)+a\calH+\calH \quad \text{and}
\quad bcb \in \Vect(b,ab)+a\calH+\calH,$$ and we conclude that
$$aca \in \Vect(\id_V,a,ba,ca)+a \calH+\calH \quad \text{and} \quad
acb \in \Vect(\id_V,a,b,ab)+a \calH+\calH.$$
This yields $a(ca(W)) \subset \overline{W}$ and $a(cb(W)) \subset \overline{W}$, and we conclude that
$a$ stabilizes $\overline{W}$.
Since $a$ is quadratic, it follows that $a^{-1}$ also stabilizes $\overline{W}$.

Now, we use an invariance trick to see that $c$ also stabilizes $\overline{W}$.
The inverse of $abc$ reads $\id_V+w'$ for some $w' \in \End(V)$. Then $w+w'+ww'=0=w'w+w+w'$, which leads to $\im w' \subset \im w$
and $\im w \subset \im w'$, and then to $\im w=\im w'$. We also note that
$c^{-1}b^{-1}a^{-1}=\id_V+w'$, and that $c^{-1}$, $b^{-1}$ and $a^{-1}$ are quadratic.
In this situation, we deduce from the above that the subspace
$$\overline{W}':=W+a^{-1}(W)+b^{-1}(W)+c^{-1}(W)+\underset{(e,f) \in \{a^{-1},b^{-1},c^{-1}\}^2}{\sum} ef(W)$$
is stable under $c^{-1}$. However, since $a,b,c$ are quadratic, one sees that $\overline{W}' \subset \overline{W}$,
and symmetrically $\overline{W} \subset \overline{W'}$. Therefore, $\overline{W}$ is stable under $c^{-1}$, and hence under
$c$ because $c$ is quadratic.

Noting that $\overline{W}$ is stable under $abc=\id_V+w$ because it includes $\im w$, we finally obtain that
$\overline{W}$ is stable under $b=a^{-1}(abc) c^{-1}$. This completes the proof.
\end{proof}

\subsection{Proof of Theorem \ref{dominanteigenvalueCN}}

Here, we derive Theorem \ref{dominanteigenvalueCN} from the preceding two results.
Let $p_1,p_2,p_3$ be split non-derogatory polynomials of $\F[t]$ with degree $2$,
and let $u$ be an endomorphism of an infinite-dimensional vector space $V$.
Assume that $u$ has a dominant eigenvalue $\lambda$ and that there exist endomorphisms $a,b,c$ of $V$
such that $u=abc$ and $p_1(a)=p_2(b)=p_3(c)=0$.

Set $w:=u-\lambda\,\id_V$.
By Lemma \ref{invariantsubspacelemma} applied to $W:=\im w$, there exists a linear subspace $\overline{W}$ of
$V$ that includes $\im w$, is stable under $a$, $b$ and $c$, and whose dimension is less than the one of $V$.
It follows that $a,b,c$ induce endomorphisms of the infinite-dimensional quotient space $\overline{V}:=V/\overline{W}$ whose
product equals $\lambda\id_{\overline{V}}$ and that are annihilated by $p_1$, $p_2$ and $p_3$, respectively.
By Proposition \ref{homothetieprop2}, we deduce that $\lambda$ is acceptable for $(p_1,p_2,p_3)$.
This completes the proof of Theorem \ref{dominanteigenvalueCN}.

\section{The case of the sum of a scalar multiple of the identity with a finite-rank endomorphism}\label{finiteranksection}

In this section, we give a partial result for the problem of determining when
an \emph{automorphism} of the form $\lambda \id_V+w$, where $\lambda$ is a scalar and $w$ is a finite-rank endomorphism of $V$,
is a $(p_1,p_2,p_3)$-product.

The following definition is relevant to this problem:

\begin{Def}
Let $A$ be an invertible matrix with entries in $\F$.
Let $p_1,p_2,p_3$ be non-derogatory split polynomials with degree $2$ over $\F$.
Let $\lambda \in \F$ be acceptable for $(p_1,p_2,p_3)$.
We say that $A$ is \textbf{a $(p_1,p_2,p_3)$-product $\lambda$-stably} if
there exists a non-negative integer $q$ such that the block-diagonal matrix $A \oplus \lambda I_q$ is a $(p_1,p_2,p_3)$-product.
\end{Def}

Note that, in the course of the proof of Proposition \ref{homothetieprop1}, we have shown that if $\lambda$ is acceptable for $(p_1,p_2,p_3)$,
then one of the matrices $\lambda I_1$ and $\lambda I_2$ is a $(p_1,p_2,p_3)$-product.
Hence, if we have an integer $q$ such that the block-diagonal matrix $A \oplus \lambda I_q$ is a $(p_1,p_2,p_3)$-product,
then $A \oplus \lambda I_{q+r}$ is a $(p_1,p_2,p_3)$-product for every even positive integer $r$
(and also for every odd positive integer $r$ if $\lambda$ is a $(p_1,p_2,p_3)$-product in $\F$).

Next, to any finite-rank endomorphism $w$ of $V$ can be attached a similarity class of square matrices as follows:
we choose a minimal (finite-dimensional) linear subspace $W$ of $V$ such that $\im w \subset W$ and $W+\Ker w=V$.
Then the similarity class of matrices that is attached to the induced endomorphism $w_{|W}$ does not depend on the choice
of $W$: we denote this similarity class by $[w]$. Moreover, if $W'$ is an arbitrary finite-dimensional linear subspace of
$V$ such that $\im w \subset W'$ and $W'+\Ker w=V$, any matrix that represents $w_{|W'}$
is similar to $M \oplus 0_q$ for some $M \in [w]$ and some non-negative integer $q$.

With that in mind, we can state a partial result:

\begin{theo}\label{theofiniterank}
Let $V$ be an infinite-dimensional vector space,
$w$ be a finite-rank endomorphism of $V$, and $\lambda$ be a nonzero scalar.
Let $p_1,p_2,p_3$ be non-derogatory split polynomials of $\F[t]$ with degree $2$.
Choose a matrix $A$ in $[w]$ with $n$ rows.
Then the following conditions are equivalent:
\begin{enumerate}[(i)]
\item The endomorphism $\lambda \id_V+w$ is a $(p_1,p_2,p_3)$-product.
\item The scalar $\lambda$ is acceptable for $(p_1,p_2,p_3)$, and the
matrix $A+\lambda I_n$ is a $(p_1,p_2,p_3)$-product $\lambda$-stably.
\end{enumerate}
\end{theo}

\begin{proof}
Set $u:=\lambda \id_V+w$.

Assume first that condition (ii) holds.
Choose a non-negative integer $q$ such that
$(A +\lambda I_n) \oplus \lambda I_q$ is a $(p_1,p_2,p_3)$-product.
We have a finite-dimensional linear subspace $W$ of $V$ such that $\im w \subset W$, $W+\Ker w=V$,
and $A$ represents the endomorphism of $W$ induced by $w$.
Then we can split $V=W \oplus W'$ where $W' \subset \Ker w$.
Since $W'$ is infinite-dimensional, we can further split $W'=W'_1 \oplus W'_2$ so that $W'_1$ has dimension $q$.
Hence, $V=W \oplus W'_1 \oplus W'_2$ and $W'_2$ is infinite-dimensional.
Since the matrix $(A+\lambda I_n) \oplus \lambda I_q$ is a $(p_1,p_2,p_3)$-product,
the automorphism $u_{|W \oplus W'_1}$ is a $(p_1,p_2,p_3)$-product.
Moreover, by Proposition \ref{homothetieprop1}, the
endomorphism $\lambda \id_{W'_2}$ of $W'_2$ is a $(p_1,p_2,p_3)$-product.
Since $u_{|W'_2}=\lambda \id_{W'_2}$, we deduce from Remark \ref{directsumremark} that
$u$ is a $(p_1,p_2,p_3)$-product.

Conversely, assume that condition (i) holds.
By Theorem \ref{dominanteigenvalueCN}, we already know that $\lambda$ is acceptable for $(p_1,p_2,p_3)$.

Let $a,b,c$ be endomorphisms of $V$ such that
$u=abc$ and $p_1(a)=p_2(b)=p_3(c)=0$.
Choosing a complementary subspace $G$ of $\Ker w$ in $V$ and applying the Invariant Subspace Lemma to the finite-dimensional subspace
$W:=\im w+G$, we obtain a finite-dimensional linear subspace $W'$
of $V$ that is stable under $a,b,c$, includes $W$ and satisfies $W'+\Ker w=V$.
Then we can split $V=W' \oplus V'$ with $V' \subset \Ker w$.
It follows that, for some non-negative integer $q$, the matrix
$(A+\lambda I_n) \oplus (\lambda I_q)$ represents the endomorphism $u'$ of $W'$ induced by $u$.
Since $a,b,c$ stabilize $W'$, we see that $u'$ is the product of the respective induced endomorphisms $a',b',c'$.
Since $p_1(a')=p_2(b')=p_3(c')=0$, it turns out that $u'$ is a $(p_1,p_2,p_3)$-product. In turn, this shows that
$(A+\lambda I_n) \oplus (\lambda I_q)$
is a $(p_1,p_2,p_3)$-product in the matrix algebra $\Mat_{n+q}(\F)$.
Hence, $A+\lambda I_n$ is a $(p_1,p_2,p_3)$-product $\lambda$-stably.
\end{proof}

Given a nonzero scalar $\lambda$ that is acceptable for $(p_1,p_2,p_3)$, the determinant yields a
natural obstruction against being a $(p_1,p_2,p_3)$-product $\lambda$-stably.
Let indeed $A \in \Mat_n(\F)$ be a $(p_1,p_2,p_3)$-product $\lambda$-stably.
Obviously the determinant of $A$ must belong to the multiplicative semigroup generated by
$\lambda$ and by the respective roots of $p_1,p_2,p_3$. In particular:
\begin{itemize}
\item If $p_1=t^2-1$ and $p_2=p_3=(t-1)^2$ or $p_2=p_3=t^2-1$, then
either $\lambda\in \{-1,1\}$ and the determinant of $A$ is required to be $\pm 1$,
or $\lambda$ has order $4$ and the determinant of $A$ is required to belong to the subgroup generated by $\lambda$.
\item If $p_1=p_2=p_3=(t-1)^2$, then either $\lambda=1$ and the determinant of $A$ must equal $1$, or $\lambda=-1$
and the determinant of $A$ must equal $\pm 1$.
\item If $p_1=p_2=t^2-1$ and $p_3=(t-1)^2$, then the determinant of $A$ must equal $\pm 1$.
\end{itemize}
In those three (actually, four!) cases, it has been shown with much difficulty in \cite{dSP3involutions} that those conditions
are actually \emph{sufficient} for $A$ to be a $(p_1,p_2,p_3)$-product $\lambda$-stably.
This leads to the following three results:

\begin{theo}
Let $\lambda \in \F \setminus \{0\}$ be such that $\lambda^4=1$.
Let $A \in \Mat_n(\F)$. The following conditions are equivalent.
\begin{enumerate}[(i)]
\item $A$ is a $(t^2-1,t^2-1,t^2-1)$-product $\lambda$-stably;
\item $A$ is a $\bigl(t^2-1,(t-1)^2,(t-1)^2\bigr)$-product $\lambda$-stably;
\item $\det A$ belongs to the multiplicative subgroup of $\F^*$ generated by $-1$ and $\lambda$.
\end{enumerate}
\end{theo}

\begin{theo}
Let $A \in \Mat_n(\F)$ and $\lambda \in \{1,-1\}$.
The matrix $A$ is a $\bigl((t-1)^2,(t-1)^2,(t-1)^2\bigr)$-product $\lambda$-stably if and only $\det A$ is a power of $\lambda$.
\end{theo}

\begin{theo}
Let $A \in \Mat_n(\F)$ and let $\lambda \in \{1,-1\}$.
Then $A$ is a $\bigl((t-1)^2,t^2-1,t^2-1\bigr)$-product $\lambda$-stably if and only $\det A=\pm 1$.
\end{theo}

Combining those three results with Theorems \ref{dominanteigenvalueCN}, \ref{theo3}, \ref{dominanteigenvalueCS} and \ref{theofiniterank} yields Theorems
\ref{theo3invol} to \ref{theo3mixed}.

In general, the condition that the determinant must belong to the multiplicative semigroup generated by the roots of the $p_i$'s and by
$\lambda$ is far from sufficient, even in the normalized situation where $1$ is a root of each polynomial $p_i$.
For example, assume that the multiplicative group $\F^*$ contains an element $x$ with infinite order, and take $p_1=p_2=p_3=(t-1)(t-x^2)$
and $\lambda=x^3$. Choose $A \in \GL_n(\F)$ having determinant $x^{12 n}$ and that does not have $x^6$ in its spectrum.
Assume that $A$ is a $(p_1,p_2,p_3)$-product $\lambda$-stably, and take a non-negative integer $r$ such that
$M:=A \oplus \lambda I_r$ is a $(p_1,p_2,p_3)$-product. Denote by $u : X \mapsto MX$ the corresponding automorphism of
$\F^{n+r}$. Then, applying the Invariant Subspace Lemma to $W:=\F^n \times \{0\}$
yields a linear subspace $W'$ of $\F^{n+r}$ that includes $W$, has dimension at most $8 n$,
is stable under $u$ and is such that the automorphism $u_{|W'}$
is a $(p_1,p_2,p_3)$-product. Setting $q:=\dim W'-n$, we see that $A \oplus \lambda I_q$
represents $u_{|W'}$ in some basis, and hence $A \oplus \lambda I_q$ is a $(p_1,p_2,p_3)$-product. Moreover,
$q \leq 7n$.

Next, let us take matrices $M_1,M_2,M_3$ of $\Mat_{n+q}(\F)$ such that $A \oplus \lambda I_{q}=M_1M_2M_3$ and $p_1(M_1)=p_2(M_2)=p_3(M_3)=0$.
Noting that $x^{12 n+3 q}=\det(A \oplus \lambda I_{q})=\det M_1\,\det M_2\,\det M_3$,
we deduce that, for all $i \in \lcro 1,3\rcro$, the sum of the (geometric) multiplicities of the eigenvalue $x^2$ for the $M_i$'s equals
$\frac{12n+3q}{2}$. However, since $\lambda \neq x^6$ and $x^6$ is no eigenvalue of $A$,
the corresponding eigenspaces have trivial intersection, which classically
leads to the sum of their dimensions being at most $2(n+q)$. Hence
$12n+3q \leq 4n+4q$, leading to $q \geq 8n$. This contradicts the fact that $q\leq 7n$.

For general values of the triple $(p_1,p_2,p_3)$, we are convinced that determining (in spectral terms) which matrices
over $\F$ are $(p_1,p_2,p_3)$-products $\lambda$-stably is hopeless.

\section{On elementary automorphisms}\label{elementarysection}

\subsection{Recognizing super-elementary automorphisms}

\begin{lemma}\label{basicelementarylemma}
Let $u$ be an automorphism of an infinite-dimensional vector space $V$,
and assume that $V=\Vect(u^k(x))_{k \in \Z}$ for some vector $x \in V$.
Then $u$ is super-elementary.
\end{lemma}

\begin{proof}
It suffices to prove that $(u^k(x))_{k \in \Z}$ is linearly independent.
Note that $x \neq 0$ otherwise $V=\{0\}$. Hence, all the terms of $(u^k(x))_{k \in \Z}$ are nonzero.
Assume that $(u^k(x))_{k \in \Z}$ is not linearly independent.
This yields integers $a<b$ and scalars $\lambda_a,\dots,\lambda_b$ such that
$\underset{k=a}{\overset{b}{\sum}} \lambda_k u^k(x)=0$ and $\lambda_a\lambda_b \neq 0$.
Then
$$u^{b-a}(x)=\sum_{k=0}^{b-a-1} -\frac{\lambda_{k+a}}{\lambda_b}\, u^k(x) \quad \text{and} \quad
u^{-1}(x)=\sum_{k=0}^{b-a-1} -\frac{\lambda_{k+a+1}}{\lambda_a}\, u^k(x),$$
which yields that the linear subspace $V':=\Vect\bigl(x,u(x),\dots,u^{b-a-1}(x)\bigr)$ is stable under both $u$ and $u^{-1}$.
Hence, $V'$ contains $u^k(x)$ for all $k \in \Z$, which contradicts the assumption that $V$ is infinite-dimensional.
\end{proof}

\begin{prop}\label{elementaryprop1}
Let $V$ be a vector space.
Let $(n_i)_{i \in \N}$ be a sequence of non-negative integers.
Let $(e_{i,j})_{i \in \N, \; 0 \leq j \leq n_i}$ be a basis of $V$.
For $k \in \N$, set $V_k:=\Vect(e_{i,j})_{0 \leq i \leq k, \; 0 \leq j \leq n_i.}$
Let $u\in \GL(V)$ act as follows on that basis:
\begin{enumerate}[(i)]
\item $u(e_{i,j})$ is collinear with $e_{i,j+1}$ for all $i \in \N$ and all $j \in \lcro 0,n_i-1\rcro$;
\item For all $i \in \N$, $u(e_{i,n_i})=\lambda_i \,e_{i+1,n_{i+1}}$ mod $V_i$ for some $\lambda_i \in \F \setminus \{0\}$.
\end{enumerate}
Then $u$ is super-elementary.
\end{prop}

\begin{proof}
Note that $V$ is infinite-dimensional.
We set $x:=e_{0,n_0}$ and we prove that $V$ equals the space $V':=\Vect(u^k(x))_{k \in \Z}$.
From there, the conclusion will follow from Lemma \ref{basicelementarylemma}. Note that $V'$ is stable under $u$ and $u^{-1}$.

First, let $i \in \N$ be such that $e_{i,n_i} \in V'$.
For all $j \in \lcro 0,n_i-1\rcro$, we deduce from condition (i) that $e_{i,j}$ is collinear with
$u^{-1}(e_{i,j+1})$. Since $V'$ is stable under $u^{-1}$, we deduce by downward induction that
it contains all vectors $e_{i,n_i}$, $e_{i,n_i-1}$, \ldots, $e_{i,0}$.

In particular, the case $i=0$ shows that $V_0 \subset V'$.
Now, we proceed by induction. Let $k \in \N$ be such that $V_k \subset V'$.
In particular $V'$ contains $e_{k,n_k}$ and hence it contains $u(e_{k,n_k})$.
By condition (ii) and since $V_k \subset V'$, we deduce that $V'$ contains $e_{k+1,n_{k+1}}$.
Then by the above we further deduce that $\Vect(e_{k+1,j})_{0 \leq j \leq n_{k+1}} \subset V'$, and finally
$V'$ includes $V_k+\Vect(e_{k+1,j})_{0 \leq j \leq n_{k+1}}=V_{k+1}$.

Hence, by induction, $V_k \subset V'$ for all $k \in \N$, and we conclude that $V'=V$.
\end{proof}

We will also need the following variation of Proposition \ref{elementaryprop1}:

\begin{prop}\label{elementaryprop2}
Let $V$ be a vector space.
Let $(n_i)_{0 \leq i \leq p}$ be a finite sequence (possibly void) of non-negative integers.
Let $(e_{i,j})_{0 \leq i \leq p, \; 0 \leq j \leq n_i}$ be a family (possibly void) of vectors of $V$, and
$(f_k)_{k \in \Z}$ be another family, such that the concatenation of those families yields a basis of $V$.
For $k \in \lcro 0,p\rcro$, set $V_k:=\Vect(e_{i,j})_{0 \leq i \leq k, \; 0 \leq j \leq n_i}$.
Let $u \in \GL(V)$ acts as follows on that basis:
\begin{enumerate}[(i)]
\item $u(e_{i,j})$ is collinear with $e_{i,j+1}$ for all $i \in \lcro 0,p\rcro$ and all $j \in \lcro 0,n_i-1\rcro$;
\item $u(f_k)$ is collinear with $f_{k+1}$ for all $k \in \Z$;
\item For all $i \in \lcro 1,p-1\rcro$, $u(e_{i,n_i})=\lambda_i\, e_{i+1,n_{i+1}}$ mod $V_i$ for some $\lambda_i \in \F \setminus \{0\}$;
\item $u(e_{p,n_p})=\mu f_0$ mod $V_p$ for some $\mu \in \F \setminus \{0\}$.
\end{enumerate}
Then $u$ is super-elementary.
\end{prop}

\begin{proof}

Again, $V$ is infinite-dimensional.

Assume first that $p \geq 0$. Once more, we set
$x:=e_{0,n_0}$ and we prove that $V$ equals $V':=\Vect(u^k(x))_{k \in \Z}$.
As in the preceding proof, we obtain by finite induction that $V_p \subset V'$. Using condition (iv), we deduce that
$f_0 \in V'$. Since $V'$ is stable under both $u$ and $u^{-1}$, condition
(iii) shows by upward and downward induction that $V'$ contains $f_k$ for all $k \in \Z$.
Hence, $V=V'$.

If the family $(e_{i,j})_{0 \leq i \leq p, \; 0 \leq j \leq n_i}$ is void, we simply take $x:=f_0$, and again we see that
$V=\Vect(u^k(x))_{k \in \Z}$.
\end{proof}

\subsection{Factorizing elementary automorphisms}

We are now ready to prove Theorem \ref{theo2}.
Using Remark \ref{directsumremark}, we see that it suffices to consider the case of super-elementary
automorphisms. Moreover, since any two super-elementary
automorphisms of a vector space $V$ are obviously conjugate to one another in the group $\GL(V)$, it suffices
to prove the following result.

\begin{prop}\label{theo2superelem}
Let $V$ be a vector space with infinite countable dimension, and let $p$ and $q$
be non-derogatory polynomials with degree $2$, with $p$ split.
Then there exist automorphisms $a,b$ of $V$ such that $p(a)=0$, $q(b)=0$, and $ab$
is super-elementary.
\end{prop}

The proof will use a technique that will be reinvested in the next section
(see the proof of Proposition \ref{goodstratprop}).

\begin{proof}
We lose no generality in assuming that $p$ and $q$ are monic. We write $p(t)=(t-\alpha)(t-\beta)$ and
$q(t)=t^2-\lambda t-\mu$.

We choose a basis $(x_n)_{n \in \N}$ of $V$ and we define two endomorphisms $a$ and $b$ of $V$ on that basis as follows:
for all $n \in \N$,
\begin{align*}
a(x_{2n})=\alpha x_{2n}+x_{2n+3}, & \quad a(x_{2n+1})=\beta x_{2n+1}, \\
b(x_{2n})=x_{2n+1}, & \quad b(x_{2n+1})=\mu x_{2n}+\lambda x_{2n+1}.
\end{align*}
For all $n \in \N$, the endomorphism $q(b)$ vanishes at $x_{2n}$ and hence it also vanishes at $x_{2n+1}=b(x_{2n})$ since it commutes with $b$.
Therefore, $q(b)=0$.
On the other hand, we see that
$$\im(a-\alpha \id_V) \subset \Vect(x_{2n+1})_{n \in \N} \subset \Ker(a-\beta \id_V),$$
and hence $p(a)=0$.

Since $p$ and $q$ are non-derogatory, it follows that $a$ and $b$ are automorphisms of $V$, and hence so is $u:=ab$.

Now we use Proposition \ref{elementaryprop1} to see that $u$ is super-elementary.
Set indeed $e_{n,i}:=x_{2n+i}$ for all $n \in \N$ and all $i \in \{0,1\}$, so that $(e_{n,i})_{n \in \N, 0 \leq i \leq 1}$
is a basis of $V$.
For all $n \in \N$, we have
$$u(e_{n,0})=a(x_{2n+1})=\beta x_{2n+1}=\beta e_{n,1},$$
whereas
$$u(e_{n,1})=\mu a(x_{2n})+\lambda a(x_{2n+1})=\mu x_{2n+3}+\mu \alpha x_{2n}+\lambda \beta x_{2n+1},$$
which equals $\mu e_{n+1,1}$ modulo $\Vect(e_{k,i})_{0 \leq k \leq n,i \in \{0,1\}}$.
We deduce from Proposition \ref{elementaryprop1} that $u$ is super-elementary.
\end{proof}

Hence, the proof of Theorem \ref{theo2} is now complete.

Let us close this section by showing that in Theorem \ref{theo2}, we cannot in general allow \emph{both} polynomials $p_1$ and $p_2$ to be irreducible.

\begin{prop}
Let $u$ be a super-elementary automorphism of a real vector space $V$.
Then $u$ is not a $(t^2+1,t^2+1)$-product.
\end{prop}

\begin{proof}
Assuming otherwise, there are automorphisms $a,b$ of $V$ such that $u=ab$ and $a^2=b^2=-\id_V$.
By Lemma \ref{commutelemma2}, those automorphisms commute with
$v:=u+u^{-1}$. We deduce that $a$ and $b$ stabilize the subspace $W:=\im v$, and so do $a^{-1}$ and $b^{-1}$.
It follows that $a$, $b$ and $u$ respectively induce automorphisms $a',b',u'$ of the quotient space
$E:=V/W$, such that $(a')^2=-\id_E=(b')^2$, and $u'=a'b'$.

Next, choose a basis $(e_k)_{k \in \Z}$ of $V$ such that $u(e_k)=e_{k+1}$ for all $k \in \Z$,
and note that $V/W$ has dimension $2$ and is spanned by the respective classes
$x$ and $y$ of $e_0$ and $e_1$. We see that $u'(x)=y$ and $u'(y)=-x$, whence
$u'$ is cyclic with characteristic polynomial $t^2+1$.
Besides, we have just seen that $u'$ is a $(t^2+1,t^2+1)$-product.
However, this runs counter to the known classification of $(t^2+1,t^2+1)$-products in the finite-dimensional case
over the field of real numbers (see section 4.4 of \cite{dSPproductquad}). Alternatively, one sees that
$a'b'a'b'=-\id_E$, which reads $a'b'=-b'a'$; combining this with $(a')^2=-\id_E=(b')^2$ would yield a linear representation
of the real algebra of quaternions on the $2$-dimensional real
vector space $E$, and classically there is no such representation.
\end{proof}

\section{Stratifications}\label{stratificationsection}

\subsection{Definition and basic considerations}

Given a non-negative integer $d$, we denote by $\F_d[t]$ the linear subspace of $\F[t]$
spanned by $(t^k)_{0 \leq k \leq d}$. We convene that $\F_{-1}[t]=\{0\}$.

Let $V$ be an $\F[t,t^{-1}]$-module.
First, let $x \in V \setminus \{0\}$ and assume that the $\F[t,t^{-1}]$-submodule generated by $x$ is not free.
Choosing $f(t) \in \F[t,t^{-1}] \setminus \{0\}$ such that $f(t)\cdot x=0$, we can factorize $f(t)=t^k p(t)$
for some $k \in \Z$ and some $p(t) \in \F[t] \setminus \{0\}$, and hence $p(t)\cdot x=0$.
As seen in the proof of Lemma \ref{basicelementarylemma}, this shows that $\F[t]x$ is stable under $x \mapsto t^{-1}x$,
and hence $\F[t,t^{-1}]x=\F[t]x$. It follows further that $\F[t,t^{-1}]x$ has finite dimension over $\F$,
and if we denote this dimension by $d$, the family $(t^k x)_{0 \leq k<d}$ is a basis of the $\F$-vector space
$\F[t,t^{-1}]x$; moreover in that situation $t^d x$ is a linear combination of $x,t x,\dots,t^{d-1}x$
with \emph{nonzero} coefficient on $x$.

In contrast, for all $x \in V \setminus \{0\}$, if the $\F[t,t^{-1}]$-submodule generated by $x$ is free then $(t^kx)_{k \in \Z}$
is a basis of it as a vector space over $\F$.

The following definition echoes a concept that was studied in \cite{dSPidempotentinfinite1,dSPidempotentinfinite2}
for modules over $\F[t]$.

\begin{Def}
Let $V$ be an $\F[t,t^{-1}]$-module.
A \textbf{stratification} of $V$ is an increasing sequence $(V_\alpha)_{\alpha \in D}$, indexed over a well-ordered set $(D,\leq)$,
of submodules of $V$ in which:
\begin{itemize}
\item For all $\alpha \in D$, the quotient module $V_\alpha/\biggl(\underset{\beta<\alpha}{\sum} V_\beta\biggr)$
is nonzero and has a generator.
\item One has $V=\underset{\alpha \in D}{\sum} V_\alpha$.
\end{itemize}
To any such stratification, we assign the \textbf{dimension sequence} $(n_\alpha)_{\alpha \in D}$ defined by $$n_\alpha:=\dim_\F\biggl(V_\alpha/\underset{\beta<\alpha}{\sum} V_\beta\biggr)$$
(in the infinite-dimensional case, we consider this dimension to be $+\infty$, not the first infinite cardinal $\aleph_0$)
and the \textbf{lower bound sequence} $(m_\alpha)_{\alpha \in D}$ defined by
$$m_\alpha:=\begin{cases}
-1 & \text{if $n_\alpha<+\infty$} \\
-\infty & \text{otherwise}.
\end{cases}$$
\end{Def}

Let $(V_\alpha)_{\alpha \in D}$ be a stratification of $V$.
For every $\alpha \in D$, we can choose a vector $x_\alpha \in V_\alpha$ such that $V_\alpha=\F[t,t^{-1}] x_\alpha +\underset{\beta<\alpha}{\sum} V_\beta$,
and we note that:
\begin{itemize}
\item if $n_\alpha$ is finite then
$V_\alpha=\F_{n_\alpha-1}[t]\,x_\alpha \oplus \underset{\beta<\alpha}{\sum} V_\beta$, the family $(t^k x_\alpha)_{0 \leq k<n_\alpha}$
is linearly independent and $t^{n_\alpha} x=\delta x$ modulo $\Vect(t^k \alpha)_{1 \leq k<n_\alpha}+\underset{\beta<\alpha}{\sum} V_\beta$
for some \emph{nonzero} scalar $\delta$ (here we have a direct sum of vector spaces over $\F$);
\item otherwise $(t^k x_\alpha)_{k \in \Z}$ is linearly independent and
$V_\alpha=\F[t,t^{-1}]\,x_\alpha \oplus \underset{\beta<\alpha}{\sum} V_\beta$.
\end{itemize}

We will say that the \textbf{vector sequence} $(x_\alpha)_{\alpha \in D}$ is attached to $(V_\alpha)_{\alpha \in D}$.
If we have such a sequence, a transfinite induction shows that, for all $\alpha$ and $\beta$ in $D$ with $\beta<\alpha$,
the family $(t^k\, x_\delta)_{\beta \leq \delta \leq \alpha, \;  m_\delta<k <n_\delta}$ is linearly independent and
\begin{equation}\label{complet}
\begin{cases}
\biggl(\underset{\gamma<\beta}{\sum} V_\gamma\biggr) \oplus \Vect\bigl((t^k x_\delta)_{\beta \leq \delta \leq \alpha, \; m_\delta < k <n_\delta}\bigr)
& =V_\alpha \\
\biggl(\underset{\gamma<\beta}{\sum} V_\gamma\biggr) \oplus \Vect\bigl((t^k x_\delta)_{\beta \leq \delta < \alpha, \; m_\delta< k <n_\delta}\bigr)
& =\underset{\gamma<\alpha}{\sum} V_{\gamma.}
\end{cases}
\end{equation}
In particular, $(t^k\,x_\alpha)_{\alpha \in D, \; m_\alpha<  k <n_\alpha}$ is a basis of $V$.
As a special case, we get:

\begin{lemma}\label{allinfinity}
Let $V$ be an $\F[t,t^{-1}]$-module and $(V_\alpha)_{\alpha \in D}$ be a stratification of it.
Assume that the corresponding dimension sequence is constant with value $+\infty$.
Then $V$ is free.
\end{lemma}

Conversely, consider a sequence $(x_\alpha)_{\alpha \in D}$, indexed over a well-ordered set $D$,
of vectors of $V$ such that $x_\alpha \not\in \underset{\beta <\alpha}{\sum} \F[t,t^{-1}]\, x_\beta$ for all
$\alpha \in D$, and $V'=\underset{\alpha \in D}{\sum} \F[t,t^{-1}]\, x_\alpha$.
Then one sees that $\Bigl(\underset{\beta \leq \alpha}{\sum} \F[t,t^{-1}]\,x_\beta\Bigr)_{\alpha \in D}$
is a stratification of $V$ with corresponding vector sequence $(x_\alpha)_{\alpha \in D}$.

Of course, any stratification can be re-indexed over an ordinal.

\subsection{Semi-good stratifications}

\begin{Def}
Let $(V_\alpha)_{\alpha \in D}$ be a stratification of the $\F[t,t^{-1}]$-module $V$ with corresponding dimension sequence $(n_\alpha)_{\alpha \in D}$.
We say that it is a \textbf{semi-good} stratification when $D$ has no greatest element and
$n_\alpha >1$ whenever $\alpha \in D$ has a predecessor.
\end{Def}

Here is a key result for proving Theorem \ref{adjacencytheo}:

\begin{prop}\label{goodstratprop}
Let $u$ be an automorphism of an infinite-dimensional vector space $V$. Assume that the module $V^u$
has a semi-good stratification. Let $p \in \F[t]$ be a split non-derogatory polynomial with degree $2$.
Then $u$ is $p$-adjacent to an elementary automorphism.
\end{prop}

\begin{proof}
Without loss of generality, we can assume that $p$ is monic. Then we split $p(t)=(t-\lambda)(t-\mu)$.
We take a semi-good stratification $(V_\alpha)_{\alpha \in \kappa}$ of the $\F[t,t^{-1}]$-module $V^u$, where $\kappa$ is an ordinal,
we denote by $(n_\alpha)_{\alpha \in \kappa}$ the associated dimension sequence,
by $(m_\alpha)_{\alpha \in \kappa}$ the associated lower bound sequence,
and we choose a corresponding vector sequence $(x_\alpha)_{\alpha \in \kappa}$.
First of all, we construct an endomorphism $a$ of $V$ as follows on the basis
$\bigl(u^k(x_\alpha)\bigr)_{\alpha \in \kappa, \; m_\alpha<  k <n_\alpha}$:
\begin{itemize}
\item for all $\alpha \in \kappa$ such that $n_\alpha<+\infty$ ,
we put $a(x_{\alpha}):=\mu x_{\alpha}+u^{n_{\alpha+1}-1} (x_{\alpha+1})$ if $n_{\alpha+1}<+\infty$,
otherwise we put $a(x_{\alpha}):=\mu x_{\alpha}+ x_{\alpha+1}$;
\item for any other basis vector $y$, we put $a(y)=\lambda y$.
\end{itemize}
We claim that $p(a)=0$. First of all, for any basis vector $y$ that is not of the form $x_{\alpha}$ with $n_\alpha<+\infty$,
we have $y \in \Ker(a-\lambda\id_V)$, and hence $p(a)[y]=0$.
Next, let $\alpha \in \kappa$ be such that $n_\alpha<+\infty$.
If $n_{\alpha+1}<+\infty$, then $(a-\mu\id_V)(x_{\alpha}) =u^{n_{\alpha+1}-1} (x_{\alpha+1})$,
a vector which belongs to $\Ker(a-\lambda\id_V)$ because $n_{\alpha+1}>1$ (by the definition of a semi-good stratification).
Otherwise, $(a-\mu\id_V)(x_{\alpha})=x_{\alpha+1}$ belongs to $\Ker(a-\lambda\id_V)$ because $n_{\alpha+1}=+\infty$.
Therefore, $p(a)=0$.

Now, we set $v:=au$, which is an automorphism of $V$. We shall prove that $v$ is elementary.
First of all, let us look at the action of $v$ on the basis $(u^k(x_\alpha))_{\alpha \in \kappa, \; m_\alpha<  k <n_\alpha.}$
\begin{itemize}
\item For all $\alpha \in \kappa$ such that $n_\alpha=+\infty$, and all $k \in \Z$, we have
$v\bigl(u^k(x_\alpha)\bigr)=\lambda\, u^{k+1}(x_\alpha)$.
\item Let $\alpha \in \kappa$ be such that $n_\alpha<+\infty$.
For every integer $k$ such that $0 \leq k< n_\alpha-1$, we have
$$v\bigl(u^k(x_\alpha)\bigr)=\lambda\, u^{k+1}(x_\alpha).$$
Remember also that there is a nonzero scalar $\delta$ such that $u^{n_\alpha}(x_\alpha)=\delta\, x_\alpha$
modulo $\Vect(u^k(x_\alpha))_{1 \leq k< n_{\alpha}}+
\underset{\beta<\alpha}{\bigoplus} V_\beta$.
Hence, $v\bigl(u^{n_\alpha-1}(x_\alpha)\bigr)=\delta\, a(x_\alpha)$ modulo $\Vect(u^k(x_\alpha))_{1 \leq k < n_{\alpha}}+V_\alpha$
(where we use the fact that $a$ maps $V_{\beta}$ into $V_{\beta+1}$ for all $\beta \in \kappa$, and
that, if $\alpha$ has a predecessor, $a$ maps $V_{\alpha-1}$ into $V_{\alpha-1}+\Vect(u^{n_\alpha-1}(x_\alpha))$),
leading to:
$$v\bigl(u^{n_\alpha-1}(x_\alpha)\bigr)=\delta\, u^{n_{\alpha+1}-1}(x_{\alpha+1}) \quad \text{mod.}\; V_\alpha \quad \text{if $n_{\alpha+1}<+\infty$}$$
and
$$v\bigl(u^{n_\alpha-1}(x_\alpha)\bigr)=\delta\, x_{\alpha+1}  \quad \text{mod.}\; V_\alpha \quad \text{if $n_{\alpha+1}=+\infty$.}$$
\end{itemize}

In the rest of the proof, we will combine the above with Propositions \ref{elementaryprop1} and \ref{elementaryprop2} to
establish that the $\F[t,t^{-1}]$-module $V^v$ is free.
We define $D$ as the set of all $\alpha \in \kappa$ such that
either $\alpha$ has no predecessor or $\alpha$ has a predecessor and $n_{\alpha-1}=+\infty$.
Note that $D$ is a non-empty well-ordered set.

Fix $\alpha \in D$. If $n_{\alpha+k}<+\infty$ for all $k \in \N$, we set
$d_\alpha:=+\infty$; other we denote by $d_\alpha$ the least positive integer $k$ such that $n_{\alpha+k-1}=+\infty$
(the definition is correct because $\kappa$ has no greatest element).
In any case, we set
$$W_\alpha:=\sum_{0 \leq k<d_\alpha} V_{\alpha+k}=\underset{0 \leq k<d_\alpha}{\bigcup} V_{\alpha+k.}$$
First of all, we see that, for all $\alpha \in \kappa$, the automorphism
$a$ maps $V_{\alpha}$ into $V_{\alpha+1}$ if $n_\alpha<+\infty$, otherwise it maps $V_\alpha$ into itself.
Hence, for all $\alpha \in D$, we find that $a$ stabilizes $W_\alpha$.
Since $a$ is quadratic we deduce that $a^{-1}$ also stabilizes $W_\alpha$, and hence
$v^{-1}=u^{-1}a^{-1}$ stabilizes $W_\alpha$.

Now, $V$ is endowed with the $\F[t,t^{-1}]$-module structure associated with $v$ (not $u$!).
We have just proved that $(W_\alpha)_{\alpha \in D}$ is a family of submodules of $V^v$.
In order to conclude, we shall prove that it is a stratification of $V^v$ in which the dimension sequence
is constant with value $+\infty$. Lemma \ref{allinfinity} will then ensure that $v$ is elementary.

To help us, we introduce some additional notation. Let $\beta \in \kappa$.
If $\beta \not\in D$ then $\beta$ has a predecessor $\beta-1$.
As there is no infinite decreasing sequence in $\kappa$, there is a uniquely-defined element $g(\beta)\in D$ together with a
positive integer $m$ such that
$g(\beta)+m=\beta$ and $g(\beta)+k \not\in D$ for all $k \in \lcro 1,m\rcro$.
It follows from the above definition
that
$$V_\beta \subset W_{g(\beta)}.$$

Next, we see that $(W_\alpha)_{\alpha \in D}$ is increasing.
Let indeed $(\alpha,\beta)\in D^2$ be such that $\alpha<\beta$.
From the definition of $d_\alpha$, it follows that $\alpha+k<\beta$ for every integer $k$
such that $0 \leq k<d_\alpha$, and hence $W_\alpha \subset \underset{\gamma \in \kappa, \gamma<\beta}{\sum} V_\gamma
\subsetneq V_\beta \subset W_\beta$.

Finally, let us fix $\alpha \in D$. We consider the quotient module
$$E:=W_\alpha/\underset{\beta \in D, \; \beta<\alpha}{\sum} W_\beta$$
and we prove that it is nonzero and free with one generator.
To start with, we check that:
\begin{equation}
\label{eq1}
\underset{\beta \in D, \; \beta<\alpha}{\sum} W_\beta=\underset{\beta \in \kappa, \; \beta<\alpha}{\sum} V_\beta.
\end{equation}
Let indeed $\beta \in \kappa$ be such that $\beta<\alpha$.
Then $V_\beta \subset W_{g(\beta)}$ with $g(\beta)<\alpha$ and $g(\beta)\in D$.
Conversely, let $\beta \in D$ be such that $\beta<\alpha$.
Then $\beta+k<\alpha$ for all $k$ such that $0 \leq k<d_\beta$, and hence the definition of $W_\beta$ yields the inclusion
$W_\beta \subset \underset{\gamma \in \kappa, \; \gamma<\alpha}{\sum} V_\gamma$.

Combining equality \eqref{eq1} with the definition of $W_\alpha$, we deduce that
$$E=\Biggl(\sum_{0 \leq k<d_\alpha} V_{\alpha+k}\Biggr)/\Biggl(\underset{\beta \in \kappa, \; \beta<\alpha}{\sum} V_\beta\Biggr).$$
For any vector $x \in W_\alpha$, denote by $\overline{x}$ its class in the quotient vector space $E$.
Hence, there are two situations for the automorphism $\overline{v}$ of $E$ induced by $v$:
\begin{itemize}
\item If $d_\alpha=+\infty$, then $\bigl(\overline{u^l(x_{\alpha+k})}\bigr)_{k \in \N, \; 0 \leq l<n_{\alpha+k}}$
is a basis of the $\F$-vector space $E$, and $\overline{v}$ satisfies the assumptions of Proposition \ref{elementaryprop1}
with respect to that basis.
\item If $d_\alpha<+\infty$, then $\bigl(\overline{u^l(x_{\alpha+k})}\bigr)_{0 \leq k < d_{\alpha}, \; m_{\alpha+k}<l<n_{\alpha+k}}$
is a basis of the $\F$-vector space $E$, and $\overline{v}$ satisfies the assumptions of Proposition \ref{elementaryprop2}
with respect to that basis.
\end{itemize}
In particular, in any case $E$ is infinite-dimensional as a vector space over $\F$, and
it follows from one of Propositions \ref{elementaryprop1} and \ref{elementaryprop2}
that the $\F[t,t^{-1}]$-module $E$ is free, nonzero and has a generator.

We conclude that $(W_\alpha)_{\alpha \in D}$ is a stratification of $V^v$ and that its dimension sequence
is constant with value $+\infty$. By Lemma \ref{allinfinity}, we conclude that $v$ is elementary.
\end{proof}

\section{Automorphisms with no dominant eigenvalue: The uncountable-dimensional case}\label{uncountablesection}

Here, we consider the case of a vector space with uncountable dimension.
In order to prove Theorem \ref{adjacencytheo} in that special setting,
we know from Corollary \ref{goodstratprop} that it suffices to prove
the following result:

\begin{prop}
Let $V$ be a vector space with uncountable dimension, and $u$ be an automorphism of $V$
with no dominant eigenvalue. Then $V^u$ has a semi-good stratification.
\end{prop}

The proof is an easy adaptation of the one of proposition 21 from \cite{dSPidempotentinfinite1}. We include it only for the
sake of completeness.

\begin{proof}
The dimension of $V$ is an uncountable \emph{cardinal} which we denote by $\kappa$
(remember that a cardinal is an ordinal that is in bijection with no smaller ordinal).
Denote by $L$ the subset consisting of the elements of the ordinal $\kappa$ that have no predecessor.
Because $\kappa$ is uncountable, $L$ has cardinality $\kappa$, and this enables us to choose a basis $(e_\alpha)_{\alpha \in L}$ of the $\F$-vector space $V$.

The construction of a semi-good stratification is performed by transfinite induction.

Let $\alpha \in \kappa$, and assume that we have constructed a family
$(V_\beta)_{\beta< \alpha}$ of submodules of $V^u$
which is a stratification of the submodule
$W:=\underset{\beta <\alpha}{\sum} V_\beta$.
Denoting the corresponding dimension sequence by
$(n_\beta)_{\beta<\alpha}$, we also assume that, for every $\beta<\alpha$, if $\beta\not\in L$ then $n_\beta \geq 2$, otherwise $e_\beta \in V_\beta$.
Note that the set $\{\beta \in \kappa: \beta<\alpha\}$ has its cardinality less than $\kappa$;
since each quotient $V_\beta/\underset{\gamma <\beta}{\sum} V_\gamma$, for $\beta<\alpha$, has countable (possibly finite) dimension
as a vector space over $\F$, we get by transfinite induction that
$\dim V_\beta<\kappa$ for all $\beta \in \kappa$, and hence $\dim W<\kappa$.

Denote by $\overline{u}$ the automorphism of $V/W$ induced by $u$. This automorphism cannot equal $\lambda \id_{V/W}$ for some $\lambda \in \F$,
otherwise $\im(u-\lambda \id_V) \subset W$ and $\lambda$ would be a dominant eigenvalue of $u$.
Hence, the classical characterization of the scalar multiples of the identity among the endomorphisms
yields a vector $y \in V$ such that $(y,u(y))$ is linearly independent modulo $W$.

Now, we put $V_\alpha:=\F[t,t^{-1}]e_\alpha+W$ if $\alpha \in L$ and $e_\alpha \not\in W$, otherwise
we put $V_\alpha:=\F[t,t^{-1}]y+W$. In any case, the quotient module $V_\alpha/W$ has a generator and it is non-zero.
Moreover, in the second case $\dim_\F (V_\alpha/W) \geq 2$. Finally, we note that $V_\alpha$
contains $e_\alpha$ whenever $\alpha \in L$. Hence, the inductive step is climbed.

We have constructed an increasing sequence $(V_\alpha)_{\alpha \in \kappa}$ of submodules of $V^u$
that is a semi-good stratification of the sum $W:=\underset{\alpha \in \kappa}{\sum} V_\alpha$
and such that $V_\alpha$ contains $e_\alpha$ for all $\alpha \in L$.
The last property shows that $W=V$, which completes the proof.
\end{proof}

To further illustrate the specificity of the uncountable-dimensional case, we
give an example where $V^u$ has no stratification indexed over an infinite set whereas $u$ has no dominant eigenvalue.
In particular, $V^u$ has no semi-good stratification.

\begin{ex}\label{exampleofnogoodstrat}
Let $d$ be a non-negative integer.
Set $R:=\F[t,t^{-1}]$ and $L:=\F[t,t^{-1}]/(t-1)$, and consider the
$\F[t,t^{-1}]$-module $V:=R \times L^d$.

Assume that $V$ has a semi-good stratification $(V_\alpha)_{\alpha \in D}$
indexed over an infinite ordinal $D$. In particular, $D$ contains all non-negative integers.
Moreover, $V_d$ has dimension at least $d+1$ as a vector space over $\F$, and hence it is not included in
$\{0\} \times L^d$. In other words, $V_d$ contains an element of the form
$(a,b)$ where $a \in R \setminus \{0\}$ and $b \in L^d$.
Write $a=\underset{k=n}{\overset{m}{\sum}} a_k\, t^k$ for integers $n \leq m$ and scalars $a_n,\dots,a_m$ such that $a_na_m \neq 0$.
Set $N:=m-n$.
Then one deduces that every element of $V$ is equal to an element of
$\F_{N-1}[t] \times L^d$ modulo $V_d$. It follows that $V/V_d$ has finite dimension (at most $N+d$) over $\F$.
This contradicts the fact that $\dim_\F(V_{d+k}/V_d)\geq k$ for every integer $k \geq 0$.

Now, we consider the automorphism $u : x \mapsto tx$ of the $\F$-vector space $V$, and we prove that
it has no dominant eigenvalue. Assume on the contrary that it has a dominant eigenvalue $\lambda$.
Since $V$ has countable dimension as a vector space over $\F$, the endomorphism $u-\lambda\,\id_V$ must have finite rank.
This contradicts the obvious fact that the pairs $((t-\lambda)t^k,0)$ in $R \times L^d$, for $k \in \Z$, are linearly independent over $\F$.
Hence, $u$ has no dominant eigenvalue, yet $V^u=R \times L^d$ has no semi-good stratification.
\end{ex}

\section{Automorphisms with no dominant eigenvalue: The countable-dimensional case}\label{countablesection}

In this section, we complete the proof of Theorem \ref{adjacencytheo} by tackling the special
case of vector spaces with infinite countable dimension.

We have already illustrated through Example \ref{exampleofnogoodstrat} that the situation is more complicated than the
one of vector spaces with uncountable dimension.
We split the discussion into two cases, whether $V^u$ is a torsion module or not.

\subsection{When $V^u$ is a torsion module}

Let $u$ be an automorphism of a vector space $V$ with infinite countable dimension.
Assume that the resulting $\F[t,t^{-1}]$-module $V^u$ is a torsion module.
As seen in the beginning of Section \ref{stratificationsection}, this implies that the corresponding
$\F[t]$-module is a torsion module. Moreover, the
$\F[t]$-submodules of $V^u$ are exactly the $\F[t,t^{-1}]$-submodules of $V^u$.
Using proposition 22 of \cite{dSPidempotentinfinite2}, we deduce:

\begin{prop}\label{torsiongoodstrat}
Let $u$ be an automorphism of a vector space $V$ with infinite countable dimension.
Assume that $u$ has no dominant eigenvalue and that $V^u$ is a torsion $\F[t,t^{-1}]$-module.
Then $V^u$ has a semi-good stratification.
\end{prop}

Hence, by Proposition \ref{goodstratprop}, we conclude that
Theorem \ref{adjacencytheo} holds whenever $V^u$ is a torsion module over $\F[t,t^{-1}]$
and $V$ has infinite countable dimension.

\subsection{When $V^u$ is not a torsion module}

Here, the strategy of using a semi-good stratification fails, and we must resort to a slightly different method.

Let $V$ be a vector space with infinite countable dimension, and let $u \in \GL(V)$.
Assume that $V^u$ is not a torsion module over $\F[t,t^{-1}]$.

In particular, there is a non-zero free submodule of $V^u$.
Using Zorn's lemma, we find a maximal $\F[t,t^{-1}]$-independent nonempty subset $A$ of $V$.
We choose $c \in A$, and we consider the free submodules
$$W:=\underset{b \in A}{\bigoplus} \,\F[t,t^{-1}] b \quad \text{and} \quad
W':=\underset{b \in A \setminus \{c\}}{\bigoplus} \F[t,t^{-1}] b.$$
The quotient module $V/W$ is a torsion one, otherwise picking
an element outside of its torsion and writing it as the class of an element $y$ of $V$, we would get that
$A \cup \{y\}$ is independent over $\F[t,t^{-1}]$, contradicting the maximality of $A$.

If $W=V$ then we already know that $u$ is elementary: then, we choose a root $\lambda$ of $p$, we put $a:=\lambda \id_V$
and we see from Lemma \ref{basicelementarylemma} that $au$ is elementary.

\vskip 3mm
In the rest of the proof, we assume that $W \subsetneq V$. Note
then that $V/W$ has countable (possibly finite) dimension $d>0$ as a vector space over $\F$.

For all $k \in \Z$, we consider the linear subspace
$$W_k:=W' \oplus \Vect\bigl(u^l(c)\bigr)_{-\infty<l\leq k,}$$
and we note that $W=\underset{k \in \Z}{\bigcup} W_k$.

\begin{claim}
There exists a stratification $(E_k)_{k \in D}$ of $V/W$ in which the indexing ordered set $D$ is either $\N$ or a finite ordinal.
\end{claim}

From that point on, we assimilate $\N$ with the first infinite ordinal, and any finite ordinal with cardinality $n$ with the
interval $\lcro 0,n-1\rcro$ of integers.

\begin{proof}
Let us choose a basis $(e_k)_{0 \leq k<d}$ of $V/W$.

By induction, we construct a sequence $(E_k)_{k \in \N}$ of finite-dimensional submodules of $V/W$ such that $e_k \in E_k$ for all $k \in \N$
such that $k<d$, as follows.

We define $E_0$ as the submodule generated by $e_0$.
Given $k \in \N \setminus \{0\}$ and a submodule $E_{k-1}$ of $V/W$, we define $E_k$ as $E_{k-1}$ if $E_{k-1}=V/W$, otherwise
we take the least non-negative integer $j<d$ such that $e_j \not\in E_{k-1}$, and we set $E_k:=E_{k-1}+\F[t,t^{-1}]e_j$.
By induction, one shows that $e_k \in E_k$ for every non-negative integer $k<d$, and hence $V/W=\underset{k \in \N}{\bigcup} E_k$.
Note that $E_{k+1}/E_k$ has a generator for all $k \in \N$, and if it is the zero module then $E_k=V/W$. It follows that:
\begin{itemize}
\item Either the sequence $(E_k)_{k \in \N}$ is increasing, and we put $D:=\N$.
\item Or there is an integer $m\geq 0$ such that $(E_0,\dots,E_m)$ is increasing and $E_m=V/W$,
in which case we set $D:=\lcro 0,m\rcro$.
\end{itemize}
In any case, we have obtained the requested stratification of $V/W$.
\end{proof}

Let us take such a stratification $(E_k)_{k \in D}$, with its associated dimension sequence denoted by $(n_k)_{k \in D}$,
and a choice of associated vector sequence $(y_k)_{k \in D}$.
Since $V/W$ is a torsion module, each $n_k$ is finite. The vectors $y_k$ are classes of vectors of $V$ modulo $W$, and
in the next step we make a convenient choice of representatives:

\begin{claim}
There exists a family $(x_k)_{k \in D} \in V^D$ such that, for all $k \in D$,
the class of $x_k$ modulo $W$ equals $y_k$, and
$u^{n_k}(x_k)$ belongs to $W_0+\Vect(u^l(x_i))_{0 \leq i \leq k, \; 0 \leq l<n_i.}$
\end{claim}

\begin{proof}
We construct such a family by induction on $k$.

Assume that the family has been constructed up to some $k\in D$ (excluding $k$).
Set $H:=\Vect(u^l(x_i))_{0 \leq i < k, \; 0 \leq l<n_i}$.

We start with an arbitrary representative $z$ of $y_k$ in $V$, and we set
$G:=\Vect(u^l(z))_{0 \leq l<n_k}$.
By assumption, we have $u^{n_k}(z) \in W_p+H+G$ for some $p \in \Z$.

The key is to prove that $z$ can be replaced with another representative of $y_k$ so that $p$ is replaced with $p-1$.
To see this, split first $u^{n_k}(z)=z_1+z_2$ for some $z_1 \in W_p$ and some $z_2 \in H+G$.
The definition of $W_p$ shows that $z_1=u^{n_k}(z'_1)$ for some $z'_1 \in W_{p-n_k}$.
Now, we set $z':=z-z'_1$, so that $u^{n_k}(z')=z_2$ and the class of $z'$ modulo $W$ is still $y_k$.
Putting $G':=\Vect(u^l(z'))_{0 \leq l <n_k}$, we see that
$G \subset   W_{p-1}+G'$, and hence $u^{n_k}(z') \in W_{p-1}+H+G'$.

Continuing by downward induction on $p$, we conclude that $z$ can be chosen so that $u^{n_k}(z) \in W_0+H+\Vect(u^l(z))_{0 \leq l<n_k}$,
which is exactly the requested result.
\end{proof}

Now, we take a family $(x_k)_{k \in D}$ of vectors that satisfies the condition of the previous claim.
For $k \in \N \cup \{-1\}$, we set
$$Z_k:=\Vect(u^l(x_i))_{i \in D, i\leq k, \; 0 \leq l<n_i.}$$
In particular, $Z_{-1}=\{0\}$.

Note that the family that consists of all the vectors $u^k(b)$, with $k \in \Z$ and $b \in A$,
and all the vectors $u^l(x_k)$, with $k \in D$ and $l \in \lcro 0,n_k-1\rcro$, is a basis of the $\F$-vector space $V$.

\vskip 3mm
We are ready to define an automorphism $a \in \GL(V)$ such that $au$ is elementary and $p(a)=0$.
Once more, we can assume that $p$ is monic.
We choose one of its roots $\alpha$, and we write
$p(t)=t^2-\lambda t-\mu$, with $\mu \neq 0$.
We define $a \in \End(V)$ on the above basis as follows:
\begin{itemize}
\item For every $k\in D$, we put
$$a(u^{k+1}(c))=x_{k} \quad \text{and} \quad a(x_{k})=\lambda\, x_{k}+\mu\, u^{k+1}(c).$$
\item For every other vector $y$ in that basis, we put $a(y):=\alpha y$.
\end{itemize}
One easily checks that $p(a)=0$, and hence $a \in \GL(V)$. Moreover, $a$ and $a^{-1}$ stabilize $W'$.
Now, we put $v:=au$, so that $v$ stabilizes $W'$. Obviously $v$ induces
an elementary automorphism of $W'$. The following result will help us complete the proof:

\begin{claim}
The automorphism $v$ is elementary.
\end{claim}

\begin{proof}
We shall prove that $V=W' + \Vect(v^k(c))_{k \in \Z}$. Since $V/W'$ is infinite-dimensional, it will ensue that
$V=W' \oplus \Vect(v^k(c))_{k \in \Z}$, that $\Vect(v^k(c))_{k \in \Z}$ is infinite-dimensional, and finally that
$v$ is elementary.

We set $V':=W' + \Vect(v^k(c))_{k \in \Z}$, a linear subspace which is stable under both $v$ and $v^{-1}$.
We prove by induction that $V'$ includes $W_k+Z_{k-1}$ for every non-negative integer $k$.

First of all, since $v(u^k(c))=\alpha u^{k+1}(c)$ for all $k<0$, we find by downward induction that $V'$
contains $u^k(c)$ for all $k \leq 0$. Therefore, $V'$ includes $W_0=W_0+Z_{-1}$.

Now, let $k$ be a non-negative integer such that $V'$ includes $W_k+Z_{k-1}$.

If $k \not\in D$, then $k$ exceeds the greatest element of $D$ and the definition of $a$ yields that
$Z_k=Z_{k-1}$ and $v(u^k(c))=\alpha u^{k+1}(c)$; since $V'$ is stable under $v$ and contains $u^k(c)$, we deduce that
it contains $u^{k+1}(c)$, and hence $V'$ includes $W_{k+1}+Z_{k-1}=W_{k+1}+Z_k$.

Assume now that $k \in D$. Then the definition of $a$ shows that
$v(u^k(c))=x_k$, and hence $x_k \in V'$. Then we find that
$V'$ also contains $v^l(x_k)=\alpha^l u^l(x_k)$ for all $l \in \lcro 1,n_k-1\rcro$.
Hence, $V'$ includes $W_k+Z_k$. In particular, $V'$ contains $u^{n_k-1}(x_k)$.

Finally, we know from the choice of the $x_i$'s that $u(u^{n_k-1}(x_k))=\delta x_k+y$ modulo $W_0+Z_{k-1}$
for some non-zero scalar $\delta$ and some vector $y$ of $\Vect(u^l(x_k))_{0<l<n_k}$.
Applying $a$ leads to $v(u^{n_k-1}(x_k))=\delta \mu\, u^{k+1}(c)$ modulo $W_k+Z_k$.
Since $\delta \mu \neq 0$, we deduce that $V'$ contains $u^{k+1}(c)$,
and we conclude that $V'$ includes $W_{k+1}+Z_{k}$.

Hence, by induction we conclude that $V'=V$, which finishes the proof.
\end{proof}

We have just shown that $u$ is $p$-adjacent to an elementary automorphism.
This was the last remaining case in the proof of Theorem \ref{adjacencytheo}, and
this result is now fully established. This also completes the proof of Theorem \ref{theo3invol}.

\section{Applications of decompositions into the product of three quadratic automorphisms}

\subsection{Decompositions into products with four factors}

Here, we derive Theorem \ref{theo4general} from Theorem \ref{theo3}. This is based on the following simple lemma:

\begin{lemma}\label{4to3lemma}
Let $u \in \GL(V)$ have a dominant eigenvalue, and let $p\in \F[t]$ be a non-derogatory polynomial
with degree $2$. Then $u$ is $p$-adjacent to an automorphism with no dominant eigenvalue.
\end{lemma}

\begin{proof}
First of all, we let $a \in \GL(V)$ and we assume that $au$ has a dominant eigenvalue.
We claim that $a$ has a dominant eigenvalue. Indeed, let us write $au=\mu \id_V+b$ and $u=\lambda \id_V+c$ where
$(\lambda,\mu)\in (\F \setminus \{0\})^2$
and $(b,c) \in \End(V)^2$ satisfy $\rk(b)<\dim V$ and $\rk(c)<\dim V$. Then
$\lambda a+ac=\mu \id_V+b$, whence $a=\lambda^{-1}\mu \id_V+d$ for $d:=\lambda^{-1}(b- ac)$,
and we note that
$$\rk(d)=\rk(b-ac) \leq \rk(b)+\rk(ac) \leq \rk(b)+\rk(c)<\dim V.$$

Hence, in order to conclude it suffices to find an automorphism $a \in \GL(V)$ that is annihilated by $p$ but has no dominant eigenvalue.
To do this, we note that we can find a splitting $V=\underset{i \in \kappa}{\bigoplus} V_i$,
where $\kappa$ denotes the dimension of $V$, and the $V_i$'s are $2$-dimensional linear subspaces. For each $i \in \kappa$,
we choose a cyclic endomorphism $a_i$ of $V_i$ with characteristic polynomial $p$, so that
$p(a_i)=0$ and $\rk(a_i-\lambda \id_{V_i}) \geq 1$ for all $\lambda \in \F$.
 Then we define $a$ as the endomorphism of $V$ whose restriction to $V_i$ equals $a_i$ for all $i \in \kappa$.
We see that $p(a)=0$ and hence $a \in \GL(V)$. Moreover, for all $\lambda \in \F$ we
have $\im(a-\lambda \id_V)=\underset{i \in \kappa}{\bigoplus} \im(a_i-\lambda \id_{V_i})$
and we conclude that $\rk(a-\lambda \id_V) \geq \kappa$. Therefore, $au$ has no dominant eigenvalue.
\end{proof}

From here, the proof of Theorem \ref{theo4general} is straightforward. Let $u \in \GL(V)$, and let $p_1,p_2,p_3,p_4$
be non-derogatory split polynomials with degree $2$. Choose a root $\lambda$ of $p_1$.
If $u$ has no dominant eigenvalue, then neither does $\lambda^{-1}u$,
and hence it is a $(p_2,p_3,p_4)$-product. Noting that $p_1(\lambda\cdot \id_V)=0$, we deduce that $u$
is a $(p_1,p_2,p_3,p_3)$-product.
If $u$ has a dominant eigenvalue, we deduce from Lemma \ref{4to3lemma} that it is $p_1^\sharp$-adjacent to an automorphism $v$
with no dominant eigenvalue; then Theorem \ref{theo3} shows that $v$ is a $(p_2,p_3,p_4)$-product, and we conclude that
$u$ is a $(p_1,p_2,p_3,p_4)$-product.

\subsection{Three factors, with an acceptable dominant eigenvalue}\label{sufficientdominantSection}

Here, we derive Theorem \ref{dominanteigenvalueCS} from Theorem \ref{theo3}.
The proof is based upon the following general result that was obtained in \cite{dSPidempotentinfinite2}
(lemma 15 there):

\begin{lemma}[Reduction Lemma]\label{reductionlemma}
Let $u \in \End(V)$, where $V$ is an infinite-dimensional vector space.
Assume that $u$ has a dominant eigenvalue $\lambda$ and that $u-\lambda \id_V$ has infinite rank.
Then there exists a decomposition $V=V_1 \oplus V_2$
into linear subspaces that are stable under $u$ and such that:
\begin{enumerate}[(i)]
\item $V_1$ is infinite-dimensional;
\item $u_{|V_1}$ has no dominant eigenvalue;
\item $u(x)=\lambda x$ for all $x \in V_2$.
\end{enumerate}
\end{lemma}

Now, we can prove Theorem \ref{dominanteigenvalueCS}.
Let $p_1,p_2,p_3$ be split non-derogatory polynomials of $\F[t]$ with degree $2$,
and let $u$ be an automorphism of an infinite-dimensional vector space $V$.
Assume that $u$ has a dominant eigenvalue $\lambda$ that is acceptable for $(p_1,p_2,p_3)$ and for which $u-\lambda \id_V$ has infinite rank.

We can find a decomposition $V=V_1\oplus V_2$ that satisfies the conditions of Lemma \ref{reductionlemma} with respect to the pair $(u,\lambda)$.
Since both $V_1$ and $V_2$ are stable under $u$, the automorphism $u$ induces respective
automorphisms $u_1$ and $u_2$ of $V_1$ and $V_2$. On the one hand, Proposition \ref{homothetieprop1} shows that
$u_2$ is a $(p_1,p_2,p_3)$-product. On the other hand, the automorphism $u_1$ of $V_1$ induced by $u$ has no dominant eigenvalue and $V_1$ is infinite-dimensional,
and hence Theorem \ref{theo3} shows that $u_1$ is a $(p_1,p_2,p_3)$-product.
By Remark \ref{directsumremark}, we conclude that $u$ is a $(p_1,p_2,p_3)$-product.

\end{document}